\documentclass[amstex, dvopdfmx, 11pt]{amsart}
\usepackage{graphicx}
\usepackage{amssymb}
\usepackage{mathrsfs}
\usepackage{placeins}

\newtheorem{Pro}{Proposition}[section]
\newtheorem{lemma}{Lemma}[section]

\newtheorem{Mythm}{Theorem}

\theoremstyle{definition}

\newtheorem{Ex}{Example}[section]
\theoremstyle{definition}
\newtheorem{Rem}{Remark}[section]
\numberwithin{equation}{section}
\makeatletter
\@namedef{subjclassname@2020}{\textup{2020} Mathematics Subject Classification}
\makeatother

\begin{document}
\title[Thompson's groups and Teichm\"uller modular groups]{Thompson's groups and Teichm\"uller modular groups of generalized Cantor sets}
\author{
  Hiroshige Shiga    
}
\address{Emeritus Professor at Tokyo Institute of Technology (Science Tokyo) \\
and
Osaka Central Advanced Mathematical Institute \\
3-3-138\ Sugimoto, Sumiyoshi-ku \ Osaka 558-8585 \ Japan
} 
\email{shiga.hiroshige.i35@kyoto-u.jp}
\date{\today}    

\begin{abstract}
Thompson's groups, which are denoted by $F, T$ and $V$, were introduced by R. Thompson in 1965. It is known that they are related to various fields in mathematics.

In this paper, we discuss Thompson's groups $F, T$ and $V$ in Teichm\"uller modular groups of Teichm\"uller spaces of generalized Cantor sets with certain conditions by using a quasiconformal geometric method.
Moreover, we show that groups corresponding to $F$ and $T$ act properly discontinuously on such Teichm\"uller spaces.
 E. de Faria, F. P. Gardiner and W. J. Harvey show  similar results for Thompson's group $F$ and asymptotic Teichm\"uller spaces of a Cantor set of bounded type.

We also discuss the actions of Thompson's groups on other spaces which we call  \lq\lq pre-Tiechm\"uller spaces"  of generalized Cantor sets.
Especially, we show that the action of a group which is isomorphic to Thompson's group $V$ on the pre-Teichm\"uller space is not properly discontinuous while groups isomorphic to $F$ and $T$ act properly discontinuously.

Finally, we construct infinitely many distinct Teichm\"uller spaces of generalized Cantor sets where our results can apply.
\end{abstract}
\maketitle
\section{Introduction}
Thompson's groups, which are denoted by $F, T$ and $V$, were introduced by R. Thompson in 1965.
Since their introduction, it has been noticed that these groups are related to various fields in mathematics.

In terms of the asymptotic Teichm\"uller theory, E. de Faria, F. P. Gardiner and W. J. Harvey \cite{FariaGardinerHarvey} showed that Thompson's group $F$ is regarded as a subgroup of the asymptotic conformal mapping class group of a Riemann surface $R$ which is the complement of a Cantor set of bounded type.
 
In this paper, we discuss Thompson's groups $F, T$ and $V$ in the usual Teichm\"uller spaces, and realize those groups in the Teichm\"uller modular group of a Riemann surface which is the complement of a generalized Cantor set with the BRD condition given in \S 2.6 (Theorem \ref{Thm:Exact}).

Our realization is quite geometric in the quasiconformal sense.
Because of this, we may consider the action of Thompson's groups on Teichm\"uller spaces and discuss the discreteness of the action of those groups on the Teichm\"uller space.
In fact, we find subgroups of the Teichm\"uller modular group which are isomorphic to $F, T$ (Theorem \ref{Thm:Isomorphism}).
Moreover, we show that the action of the group isomorphic to Thompson's group $T$ (and $F$) is discrete (Theorem \ref{Thm:Discrete Action}).

In our discussion, Thompson's group $V$ is a different feature from $F$ and $T$.
To consider the action of $V$, we introduce a new space called the pre-Teichm\"uller space, which is also a complete metric space as the Teichm\"uller space.
We may define the Teichm\"uller modular on the space and find subgroups of the group isomorphic to Thompson's groups $F, T$ and $V$.
We show that the groups isomorphic to $F, T$ act properly discontinuously on the pre-Teichm\"uller space but the group isomorphic to $V$ does not (Theorem \ref{Thm:indiscrete}).

At the final section, we show that Thompson's groups act on infinitely many distinct Teichm\"uller spaces of generalized Cantor sets.

\section{Preliminaries and main results}
\subsection{Thompson groups}
We will give a brief introduction of Thompson's groups and exhibit some fundamental properties. See \cite{CannonFloydParry} for more details.

\medskip
\noindent
{\bf Thompson's group $F$.}
A rational number is called a \emph{dyadic rational number} if it can be written as $\frac{k}{2^n}$ for some $k, n\in \mathbb N\cup\{0\}$.
Thompson's group $F$ is the set of self-homeomorphisms $f$ of $[0, 1]$ with $f(0)=0$ which are differentiable except at finitely many dyadic rational numbers and such that the derivatives are powers of $2$ on the intervals of differentiability.

For each $f\in F$, we may associate binary trees $D_f$ for the domain of $f$ and $R_f$ for the range, which are subtrees of a binary tree $\mathcal T$ constructed as below.

An interval in $[0, 1]$ is called a \emph{standard dyadic interval} if it has the form $[\frac{k}{2^n}, \frac{k+1}{2^n}]$ for some integers $n, k$ with $k\leq 2^n-1$.
From standard dyadic intervals, we may define a tree $\mathcal T$ called a \emph{tree of standard dyadic intervals}.
A vertex of $\mathcal T$ corresponds to a standard dyadic interval.
We denote a vertex by $v_I$ if it corresponds to a standard dyadic interval $I$.
The vertex corresponding to $[0, 1]$ is denoted by $v_0$ and we call it the \emph{root} of $\mathcal T$.
Two vertices $v_I$ and $v_J$ are joined by an edge $(I, J)$ if $J$ is either the left half or the right half of $I$.
The edge is called a left edge if $J$ is the left half and a right edge otherwise.
Hence, $\mathcal T$ is a binary tree rooted at $v_0$.

A finite subtree $S$ containing the root $v_0$ of $\mathcal T$ is called an \emph{ordered rooted binary tree} if the root $v_0$ has two edges, and every vertex $v_I$ of $S$ with valence greater than $1$ has both the left and the right edges.
Vertices with valence $0$ are called \emph{leaves} of $S$.
The tree $S$ is trivial if $S=\{v_0\}$.

Let $v_{I_1}, \dots , v_{I_n}$ be the set of leaves of an ordered rooted binary tree $S$, where a leaf $I_i$ is located on the left of $I_{i+1}$ $(i=1, \dots , n-1)$.
Then, $[0, 1]=\cup_{i=1}^{n}I_i$ and they give a partition of $[0, 1]$ via standard dyadic intervals.
Conversely, a partition of $[0, 1]$ via standard dyadic intervals gives an ordered rooted binary tree such that the set of the leaves corresponds to the partition.

For $f\in F$, there exist $0=x_0<x_1<\dots <x_n=1$ such that $f$ is differentiable in $(x_i, x_{i+1})$ $(i=0, 1, \dots n-1)$ and both $\{x_0, x_1, \dots , x_n\}$ and $\{0=f(x_0), f(x_1), \dots f(x_n)=1\}$ give partitions of $[0, 1]=f([0, 1])$ via standard dyadic intervals.
We denote by $D_f$ the ordered rooted binary tree for the partition by $\{x_0, x_1, \dots , x_n\}$ and  by $R_f$ for the partition by $\{0, f(x_1), \dots , 1\}$.
We put the number $n$ for the leaf $v_{[x_{n-1}, x_n]}$ in $D_f$ and for the leaf $f(v_{[x_{n-1}, x_n]})$ $(n\in \{1, \dots , n\})$.
We call the pair $(D_f, R_f)$ a \emph{tree diagram} for $f\in F$.
If $f, g\in F$ have the same tree diagram, then $f=g$.
On the other hand, for a given $f\in F$, a tree diagram for $f$ is not unique.
We identify two tree diagrams if they determine the same $f\in F$.

Since $f\in F$ is an increasing function, the numbers associated to the leaves of $R_f$ are located in the same order from left to right as those of $D_f$. 
Such a diagram $(D, R)$ is  called  an \emph{order-preserving} tree diagram and the following is known.
\begin{Pro}
\label{Pro:Tree for F}
	Let $(D, R)$ be an order-preserving tree diagram.
	Then, there exists $f\in F$ with a tree diagram $(D_f, R_f)$ such that $(D_f, R_f)=(D, R)$ as order-preserving tree diagrams.
\end{Pro}

\begin{Ex}
\label{Ex.F}
	Consider maps $f_0, f_1\in F$ given by
	\begin{equation*}
		f_0(x)=\begin{cases}
			\frac{x}{2}, &0\leq x\leq \frac{1}{2} \\
			x-\frac{1}{4}, &\frac{1}{2}\leq x\leq\frac{3}{4} \\
			2x-1, &\frac{3}{4}\leq x\leq 1
		\end{cases}
		\quad f_1(x)=\begin{cases}
			x, &0\leq x\leq \frac{1}{2} \\
			\frac{x}{2}+\frac{1}{4}, &\frac{1}{2}\leq x\leq \frac{3}{4} \\
			x-\frac{1}{8}, &\frac{3}{4}\leq x\leq \frac{7}{8}\\
			2x-1, &\frac{7}{8}\leq x\leq 1
		\end{cases}
	\end{equation*}
Then, the tree diagrams $(D_i, R_i)$ in Figure \ref{Fig.GroupF} are associated to $f_i$ $(i=0, 1)$. 
\begin{figure}[htbp]
	\centering
	\includegraphics[width=12cm]{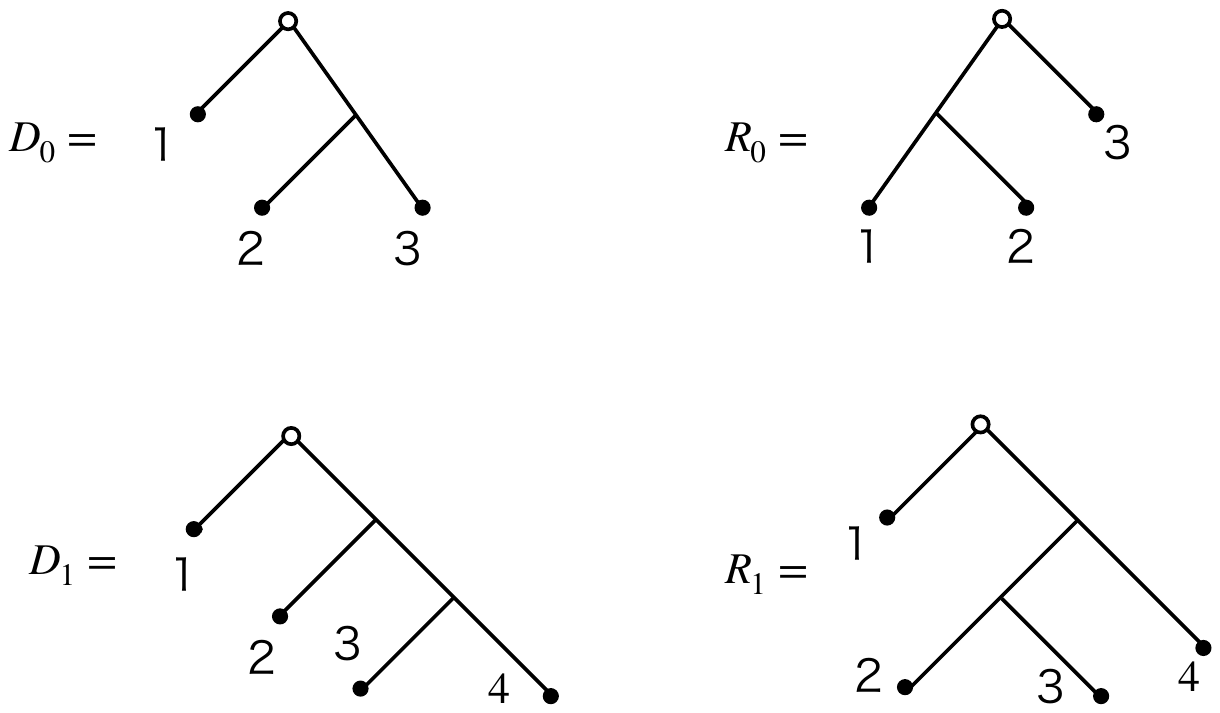}
	\caption{Group F}
	\label{Fig.GroupF}
\end{figure}
\end{Ex}

The mappings $f_0$ and $f_1$ are special ones. 
Actually, the following is known:
\begin{Pro}
\label{Pro:Generator F}
	Thompson's group $F$ is generated by $f_0$ and $f_1$.
\end{Pro}

\noindent
{\bf Thompson's group $T$.}
Next, we define Thompson's group $T$.
The unit circle $S^1$ is obtained from $[0, 1]$ by identifying the endpoints.
In other words, $S^1$ is the quotient space of the real line under the action of $x\mapsto x+n$ $(n\in \mathbb Z)$.
Thompson's group $T$ is regarded as the set of self-homeomorphisms $f$ of $\mathbb R$ with $f(0)=0$ and $f(x+n)=f(x)+n$ $(x\in \mathbb R, n\in \mathbb Z)$ which are differentiable except at finitely many dyadic rational numbers in $[0, 1]$ and such that the derivatives are powers of $2$ on the intervals of differentiability.
Obviously, $F$ is considered as a subgroup of $T$.

We may associate a tree diagram $(D_f, R_f)$ for $f\in T$ in the same way as for $F$.
We also give a number to each leaf of $D_f$ and $R_f$ in the same way as for $F$.
Because of the construction of $f\in T$, the order of the numbers associated to the leaves of $R_f$ may be a cyclic permutation of that of $D_f$.
So, we say that such a tree diagram $(D, R)$ is  \emph{positively ordered}.

\begin{Pro}
\label{Tree for T}
	Let $(D, R)$ be a positively ordered tree diagram.
	Then, there exists $f\in T$ with a tree diagram $(D_f, R_f)$ such that $(D_f, R_f)=(D, R)$ as positively ordered tree diagrams. 
\end{Pro}

\begin{Ex}
	Consider a map $f_2\in T$ on $\mathbb R$ given by
	\begin{equation*}
		f_2(x)=\begin{cases}
			\frac{x}{2}+\frac{3}{4}, &0\leq x\leq \frac{1}{2} \\
			2x, &\frac{1}{2}\leq x\leq \frac{3}{4} \\
			x+\frac{3}{4}, &\frac{3}{4}\leq x\leq 1 \\
			f_2(x-n)+n, & n\leq x\leq n+1\quad (n\not=0)
		\end{cases}
	\end{equation*}
	Then, the tree diagram $(D_2, R_2)$ in Figure \ref{Fig.GroupT} is associated to $f_2$.
	\begin{figure}[htbp]
		\centering
		\includegraphics[width=12cm]{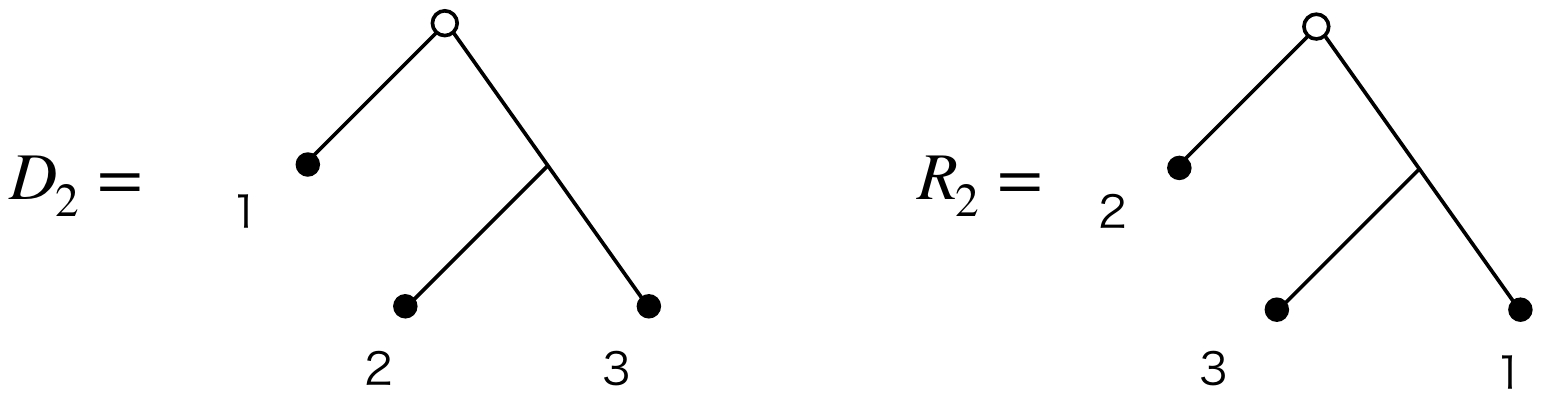}
		\caption{Group T}
		\label{Fig.GroupT}
	\end{figure}
\end{Ex}
It is known that the map $f_2$ generates Thompson's group $T$ together with $f_0$ and $f_1$.
\begin{Pro}
\label{Pro:Generator T}
	Thompson's group $T$ is generated by $f_0, f_1$ and $f_2$.
\end{Pro}

\noindent
{\bf Thompson's group $V$.}
Finally, we introduce Thompson's group $V$. While the definition of $V$ is similar to that of $T$, the group $V$ consists of right-continuous functions on standard dyadic intervals.
That is, Thompson's group $V$ consists of self-mappings $f$ of $\mathbb R$ with $f(0)=0$ and $f(x+n)=f(x)+n$ $(n\in \mathbb Z)$ satisfying
\begin{enumerate}
	\item in $[0, 1]$, $f$ is differentiable except at finitely many dyadic rational numbers, where $f$ is right-continuous;
	\item $f$ is a bijection from $[0,1]$ to itself;
	\item the derivative of $f$ on the intervals of differentiability is a power of $2$.
\end{enumerate}

Thompson's group $V$ contains Thompson's group $T$ as a subgroup, and a tree diagram $(D_f, R_f)$ is associated with $f\in V$ in the same way as for $T$.
Unlike tree diagrams for $F$ and $T$, each leaf of $D_f, R_f$ for $f\in V$ has just a number while both trees have the same number of leaves.
We call such a tree diagram a \emph{numbered diagram}.

\begin{Pro}
	Let $(D, R)$ be a numbered tree diagram.
	Then, there exists $f\in V$ with a tree diagram $(D_f, R_f)$ such that $(D_f, R_f)=(D, R)$ as numbered tree diagrams.
\end{Pro}

\begin{Ex}
	Consider a map $f_3\in V$ on $\mathbb R$ given by
	\begin{equation*}
		f_3(x)=\begin{cases}
			\frac{x}{2}+\frac{1}{2}, &0\leq x<\frac{1}{2}\\
			2x-1, &\frac{1}{2}\leq x<\frac{3}{4} \\
			x, &\frac{3}{4}\leq x<1 \\
			f_3(x-n)+n, & n\leq x<n+1 (n\not=0).
		\end{cases}
	\end{equation*}
	The tree diagram $(D_3, R_3)$ in Figure \ref{Fig.GroupV} is associated to $f_3$.
	\begin{figure}[htbp]
		\centering
		\includegraphics[width=12cm]{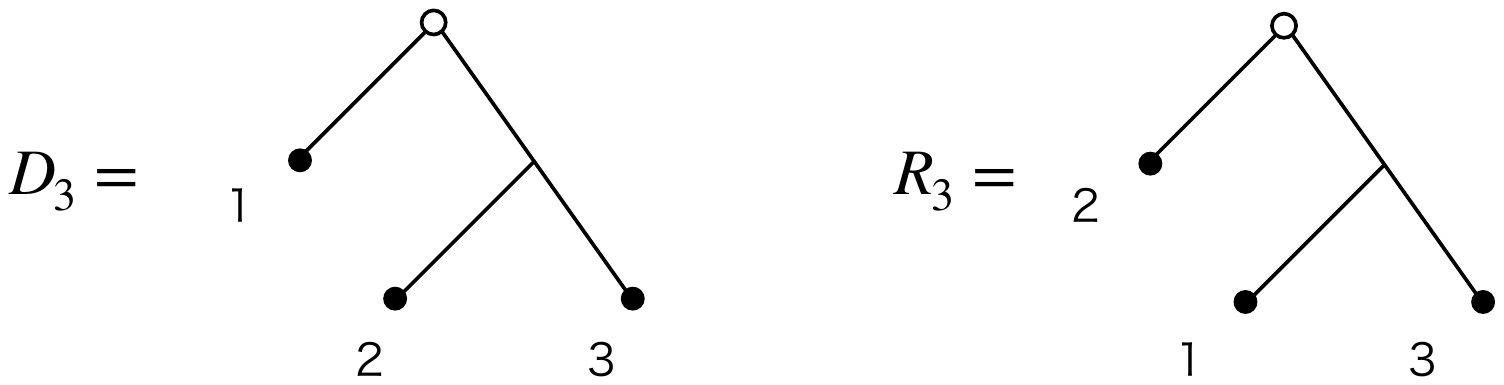}
		\caption{Group V}
		\label{Fig.GroupV}
	\end{figure}
	\FloatBarrier

\end{Ex}
It is known that Thompson's group $V$ is generated by those mappings.
\begin{Pro}
\label{Pro:Generator V}
	Thompson's group $V$ is generated by $f_0, f_1, f_2$ and $f_3$.
\end{Pro}

\noindent
{\bf The caret operation.}
As it is mentioned earlier, a map $f$ in Thompson's group is determined by different tree diagrams.
Actually, the following operation, which we call the \emph{caret operation}, makes a new tree diagram which keeps the map in Thompson's group.
Here, we mean by a \emph{caret} an ordered rooted binary tree with exactly two edges (Figure \ref{Fig. A caret}).
	\begin{figure}[htbp]
		\centering
		\includegraphics[width=12cm]{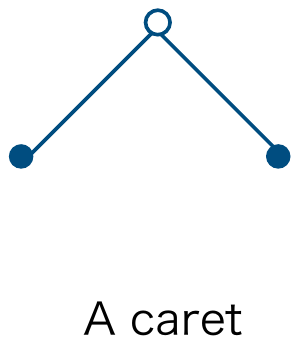}
		\caption{}
		\label{Fig. A caret}
	\end{figure}
		\FloatBarrier
Let $(D_f, R_f)$ be a numbered tree diagram associated to $f\in V$ with $N$ leaves.
Choose the $k$-th leaf $v_k$ $(k\in \{1, 2, \dots , N\})$ in $D_f$ and add a caret with the left leaf $v_k^L$ and the right leaf $v_k^R$ at $v_k$.
Then, we have a tree $D_f^k$ with $(N+1)$ leaves (Figure \ref{Fig. caret}).
	\begin{figure}[htbp]
		\centering
		\includegraphics[width=12cm]{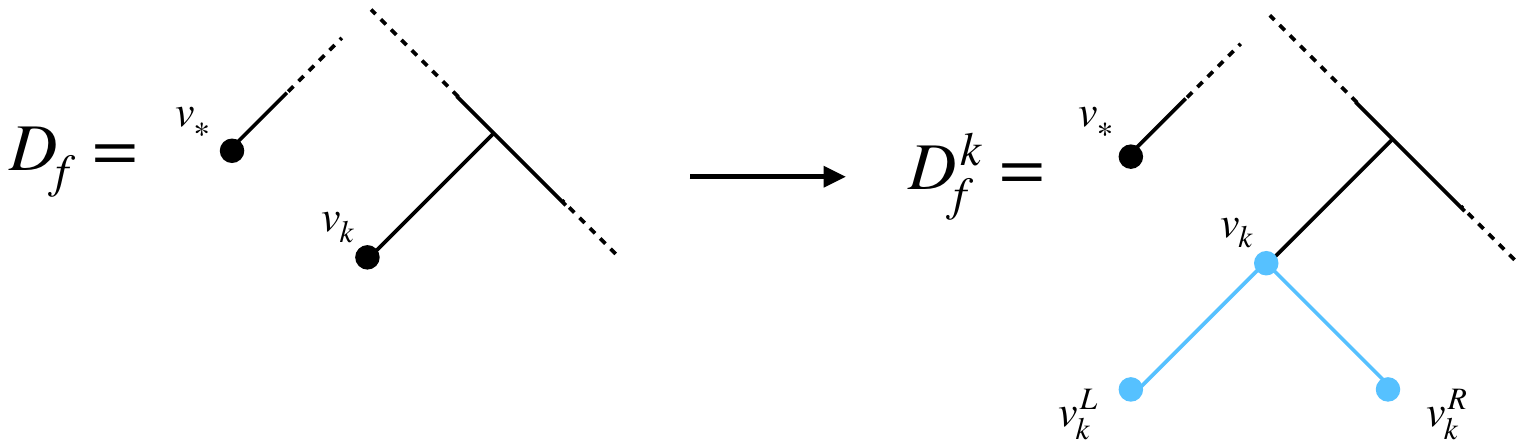}
		\caption{The caret operation}
		\label{Fig. caret}
	\end{figure}
		\FloatBarrier

We apply the same operation to $R_f$.
Namely, by adding a caret at the $k$-th leaf $v_k'$ of $R_f$, we obtain a tree $R_f^k$ with $(N+1)$ leaves.
Giving numbers to leaves of $D_f^k, R_f^k$ properly, we see that the numbered tree diagram $(D_f^k, R_f^k)$ determines the same mapping $f$ as $(D_f, R_f)$.
\medskip

\noindent
{\bf  Presentations of Thompson's groups.}
Thompson's groups $F, T$ and $V$ are finitely generated groups.
Moreover, it is known that they are finitely presented groups.
Here, we give presentations for the group $T$ for the later use in this paper.
\begin{Pro}[\cite{CannonFloydParry} Corollary 5.9]
	\label{Pro:Presentaion}
	The Thompson group $T=<f_0, f_1, f_2>$ is presented by the following relations:
	\begin{enumerate}
		\item $[f_0\circ f_1^{-1}, f_0^{-1}\circ f_1\circ f_0]=id.$,
		\item $[f_0\circ f_1^{-1}, f_0^{-2}\circ f_1\circ f_0^2]=id.$,
		\item $f_2^{-1}\circ f_1\circ (f_0^{-1}\circ f_2\circ f_1)=id.$,
		\item $((f_0^{-1}\circ f_2\circ f_1)\circ (f_0^{-1}\circ f_1\circ f_0))^{-1}\circ f_1\circ (f_0^{-2}\circ f_2\circ f_1^2)=id.$,
		\item $(f_2\circ f_0)^{-1}\circ (f_0^{-1}\circ f_2\circ f_1)^2=id.$,
		\item $f_2^3=id.$
	\end{enumerate}
	The relations (1) and (2) present $F$.
\end{Pro}

\subsection{Generalized Cantor sets and the tree}
For $\omega=(q_n)_{n=1}^{\infty}\in \Omega :=(0, 1)^{\mathbb N}$, we may define a generalized Cantor set $E(\omega)$ on $[0, 1]$.

The construction of $E(\omega)$ is similar to that of the standard middle one-third Cantor set.
We start from the unit interval $I_0=[0, 1]$.
At the first step, we remove an open interval $J_1^1$ of length $q_1=q_1|I_0|$ so that the remaining intervals $I_1^1$, $I_1^2$ have the same length $\frac{1-q_1}{2}|I_0|$, where the closed interval $I_1^1$ is located to the left of $I_1^2$.
We put $E_1(\omega)=I_1^1\cup I_1^2$.
We continue this process inductively to obtain $E_k(\omega)=\cup_{j=1}^{2^k}I_k^j$; for each closed interval $I_{k-1}^j$ $(j\in \{1, 2, \dots , 2^{k-1}\})$, we remove an open interval $J_k^{2j-1}$ of length $q_k|I_{k-1}^j|$ so that $I_{k-1}^j\setminus J_k^{2j-1}$ consists of two closed intervals $I_k^{2j-1}$, $I_k^{2j}$ of the same length, where $I_k^{2j-1}$ is located to the left of $I_k^{2j}$.
We put $E_k(\omega)=\cup_{j=1}^{2^k}I_k^j$, which is the union of $2^k$ closed intervals of the same length and $[0, 1]\setminus E_k(\omega)$ consists of $2^k-1$ open intervals $J_k^1, \dots , J_k^{2^k-1}$.
The interval $J_k^{j}$ is located between $I_k^{j}$ and $I_k^{j+1}$.
Then, $E(\omega):=\cap_{k=1}^{\infty}E_k(\omega)$ is a Cantor set on $[0, 1]$ and we call it the generalized Cantor set for $\omega\in \Omega$.

\medskip

\noindent
{\bf The canonical pants decomposition (cf. \cite{ShigaStructure}).}
We put $X(\omega)={\mathbb C}\setminus E(\omega)$ for $\omega=(q_n)_{n=1}^{\infty}\in \Omega:=(0, 1)^{\mathbb N}$.
The surface $X(\omega)$ is a hyperbolic Riemann surface of infinite type.
We define a pants decomposition of $X(\omega)$.

First, we take simple closed geodesics $\gamma_1^{j}$ surrounding $I_1^j$ $(j=1, 2)$ so that $\gamma_1^1$ and $\gamma_1^2$ together with $\{\infty\}$ give a pair of pants $P_0^1$ in $X(\omega)$. 

Next, we consider $E_2(\omega)=\cup_{j=1}^4 I_2^j$ and take circles $C_2^j$ in $X(\omega)$  surrounding $I_2^{2j-1}$ and $I_2^{2j}$ $(j=1, 2)$.
By considering the geodesics $\gamma_2^j$ in $X(\omega)$ homotopic to $C_2^j$,  we find two pairs of pants, $P_1^1$ and $P_1^2$ so that $\partial P_1^1=\gamma_1^1\cup\gamma_2^1\cup\gamma_2^2$ and $\partial P_1^2=\gamma_1^2\cup\gamma_2^3\cup\gamma_2^4$.

Inductively, we construct a pair of pants $P_k^j$ for $k\in \mathbb N$ and $j\in \{1, 2, \dots 2^k\}$.
Each $P_k^j$ is bounded by geodesics $\gamma_k^j, \gamma_{k+1}^{2j-1}$ and $\gamma_{k+1}^{2j}$, and $\{P_k^j\}_{k\in \mathbb Z_{\geq 0} , j=1, 2, \dots , 2^k}$ gives a pants decomposition of $X(\omega)$.
We call it \emph{the canonical pants decomposition} of $X(\omega)$.
Each geodesic $\gamma_k^j$ is called a \emph{pants geodesic of depth $k$} and a simple closed curve homotopic to $\gamma_k^j$ is called a \emph{pants curve of depth $k$}.

\medskip
\noindent
{\bf A tree of the canonical pants decomposition.}
From the canonical pants decomposition of $X(\omega)$, we may construct a binary tree $\mathcal T_{C}$ as follows.
The vertices of $\mathcal T_{C}$ correspond to pants geodesics and $\{\infty\}$.
We take the vertex $v_{\infty}$ corresponding to $\{\infty\}$ as the root of the tree.
The vertex corresponding to the pants geodesic $\gamma_{k}^j$ is denoted by $v_{k}^{ j}$.
The root vertex $v_{\infty}$ is joined with $v_1^1$ and $v_1^2$.
The vertex $v_k^j$ is joined with $v_{k+1}^{2j-1}$ and $v_{k+1}^{2j}$ $(k\in \mathbb N; j\in\{1, 2, \dots 2^k\})$.
The vertices $v_k^j$ and $v_{k+1}^{2j-1}$ are joined by a left edge, and the vertices $v_k^j$ and $v_{k+1}^{2j}$ are joined by a right edge.

Thus, we obtain a binary tree $\mathcal T_C$ rooted at $v_{\infty}$.
It is not hard to see that the tree $\mathcal T_C$ is isomorphic to $\mathcal T$ as rooted binary trees.
Namely, there exists a tree isomorphism $\iota_{C} : \mathcal T_C\to \mathcal T$ such that $\iota_{C} (v_{\infty})=v_0, \iota_C (v_1^1)=v_{[0, \frac{1}{2}]}$ and 
$\iota_C (v_1^2)=v_{[\frac{1}{2}, 1]}$. 
Furthermore, for an unbounded subdomain $W\subset \mathbb C$ bounded by finitely many pants geodesics with $E(\omega)\subset \mathbb C\setminus \overline W$, we may associate a subtree $T_W$ of $\mathcal T_C$ from $\overline W$ (see Figure \ref{Fig. iota}).
	\begin{figure}[htbp]
		\centering
		\includegraphics[width=12cm]{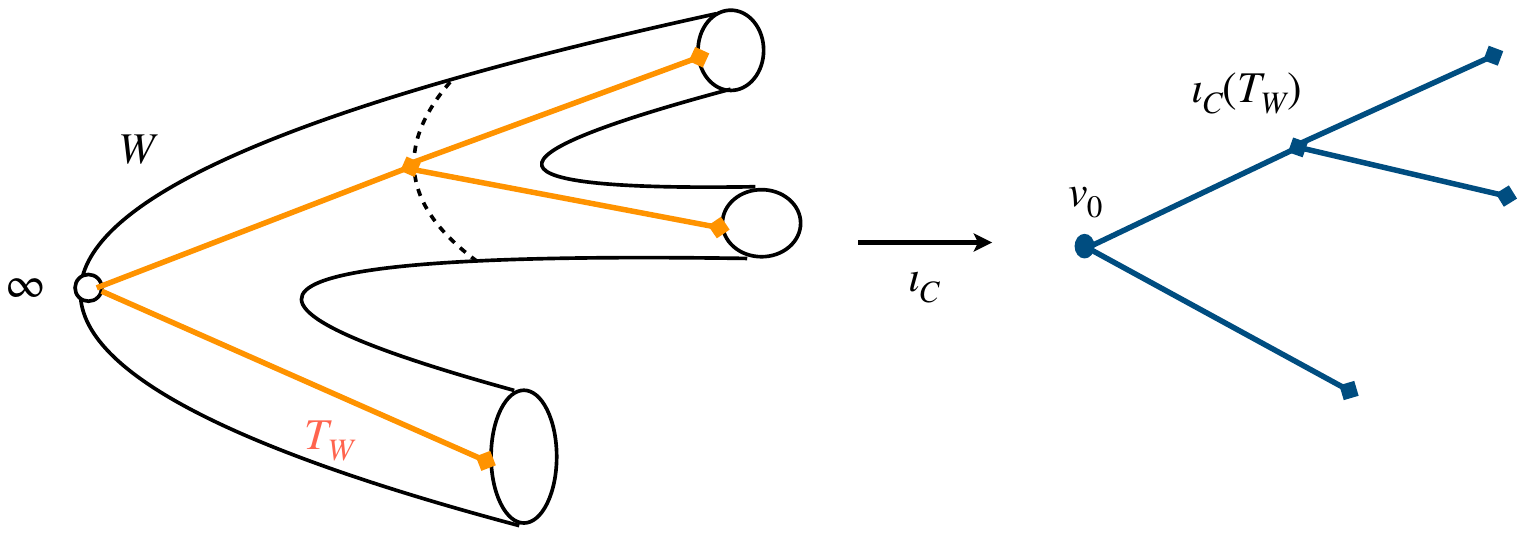}
		\caption{}
		\label{Fig. iota}
	\end{figure}
		\FloatBarrier
Since every pants geodesic contained in $W$ is a boundary curve of two pairs of pants, we may show the following.
\begin{Pro}
\label{Pro:Ordered rooted binary tree}
Let $W$ be an unbounded domain bounded by finitely many pants geodesics with $E(\omega)\subset \mathbb C\setminus \overline W$.
Then, $\iota_C (T_W)$ is an ordered rooted binary tree.
	Conversely, if $\iota_C(T)$ is an ordered binary tree for a subtree $T$ of $\mathcal T_C$, then there exists an unbounded domain $W$ in $\mathbb C$ bounded by pants curves such that $T=T_W$.
\end{Pro}

We may estimate the hyperbolic lengths of pants geodesics in $X(\omega)$ (cf. Kinjo \cite{Kinjo}, Shiga \cite{ShigaStructure} (5.2)).
\begin{Pro}
\label{Pro:Length to 0}
	Let $\gamma_d$ be a pants geodesic of depth $d$ in $X(\omega)$.
	Then,
	\begin{equation*}
		\ell_{X(\omega)}(\gamma_d)\leq \frac{2\pi^2}{\log\left (1+ \frac{2\delta}{1-q_d}\right )},
	\end{equation*}
	where $\ell_{R}(\gamma)$ stands for the hyperbolic length of a curve $\gamma\subset R$ with respect to the hyperbolic metric on a hyperbolic Riemann surface $R$ .
	In particular, $\lim_{d\to \infty}\ell_{X(\omega)}(\gamma_d)=0$ if $\lim_{n\to\infty} q_n=1$.
\end{Pro}

\subsection{Teichm\"uller space and Teichm\"uller modular group}
We will give brief explanations of Teichm\"uller spaces, quasiconformal mappings and Teichm\"uller modular groups.
The reader is referred to Ahlfors \cite{Ahlfors}, Lehto-Virtanen \cite{Lehto-Virtanen},  Imayoshi-Taniguchi \cite{Imayoshi-Taniguchi} and Gardiner-Lakic (\cite{GardinarLakic}) for details.

Let $X_0$ be a hyperbolic Riemann surface.
A pair $(X, f)$ of a Riemann surface $X$ and a quasiconformal mapping $f : X_0\to X$ is called a \emph{marked Riemann surface}.
We say that two marked Riemann surfaces $(X_i, f_i)$ $(i=1, 2)$ are \emph{Teichm\"uller equivalent}  if there exists a conformal mapping $F : X_1\to X_2$ such that $f_2\circ f_1^{-1}$  is homotopic to $F$ with respect to a homotopy keeping any point at the ideal boundary fixed.
We denote by $[X, f]$ the Teichm\"uller equivalence class of $(X, f)$ and the set of the Teuchm\"uller equivalence classes of marked Riemann surfaces is called \emph{the Teichm\"uller space of $X_0$} and it is denoted by $\textrm{Teich}(X_0)$.

From now on, we say that two quasiconformal mappings on a Riemann surface $X$ are \lq\lq homotopic" if there exists a homotopy between the two mappings keeping any point at the ideal boundary of $X$ fixed.

\begin{Rem}
\label{Remark:Koebe}
	If $X$ is a Riemann surface of genus $0$, then there exists a conformal mapping $F$ from $X$ to a domain in $\mathbb C$ (Koebe's uniformization theorem).
	Hence, if $X_0$ is a planar domain, then for a marked Riemann suface $(X, f)$ we may assume that $X$ is a domain in $\mathbb C$ and $f : X_0\to X$ is a quasiconformal mapping from $X_0$ into $\mathbb C$.
\end{Rem}

The Teichm\"uller space $\textrm{Teich}(X_0)$ has a natural complex structure, while it is infinite-dimensional if $X_0$ is of topologically infinite type.
The Teichm\"uller space $\textrm{Teich}(X_0)$ also has a natural complete distance called \emph{the Teichm\"uller distance} $d_T$ defined by
\begin{equation*}
	d_T ([X_1, f_1], [X_2, f_2])=\frac{1}{2}\inf_{f}\log K(f),
\end{equation*}
where $K(f)=(1+\|\mu_f\|_{\infty})(1-\|\mu_f\|_{\infty})^{-1}$  for $\mu_f=f_{\bar z}/f_z$,  the maximal dilatation of $f$, and the infimum is taken over all quasiconformal mappings $f$ homotopic to $f_2\circ f_1^{-1}$.

The \emph{Teichm\"uller modular group} $\textrm{Mod}(X_0)$ is the set of homotopy classes $[\varphi]$ of quasiconformal self-mappings $\varphi$ of $X_0$.
The Teichm\"uller modular group $\textrm{Mod}(X_0)$ acts on $\textrm{Teich}(X_0)$ by
\begin{equation*}
	\chi_{[\varphi]}([X, f])=[X, f\circ \varphi^{-1}].
\end{equation*}

It is easy to see that the action is isometric with respect to the Teichm\"uller distance.
If the dimension of the Teichm\"uller space is finite, then the action of the Teichm\"uller modular group is properly discontinuous.
However, if the Teichm\"uller space is infinite-dimensional, then the action is not properly discontinuous in general. 

Here, we state a fundamental fact concerning hyperbolic geometry and quasiconformal mappings, which we frequently refer to as \lq\lq Wolpert's lemma\rq\rq\ in this paper.

\begin{Pro}[\cite{Wolpert}]
\label{Pro:Wolpert}
	Let $\varphi : R\to R'$ be a quasiconformal mapping between hyperbolic Riemann surfaces $R, R'$.
	Then, for any closed hyperbolic geodesic $\gamma\subset R$, we have
	\begin{equation*}
		\frac{1}{K(\varphi)}\ell_R (\gamma)\leq \ell_{R'}([\varphi(\gamma)])\leq K(\varphi)\ell_{R}(\gamma),
	\end{equation*}
	where $[\varphi (\gamma)]$ is the geodesic homotopic to $\varphi (\gamma)$.
\end{Pro}

\subsection{Teichm\"uller modular group of generalized Cantor set}
Let $E(\omega)$ be a generalized Cantor set for $\omega \in \Omega$ and $X(\omega)=\mathbb C\setminus E(\omega)$.

\begin{lemma}
\label{lemma:QC-extension}
	Let $\varphi : X(\omega)\to \mathbb C$ be a quasiconformal mapping on $X(\omega)$.
	Then, $\varphi$ extends to a quasiconformal mapping on $\mathbb C$. 
\end{lemma}
\begin{proof}
	Let $\textrm{End}(X)$ be the set of topological ends of a subdomain $X$ of $\mathbb C$.
	Since $\varphi : X(\omega)\to \varphi (X(\omega))$ is a homeomorphism, the mapping extends a homeomorphism $\widehat\varphi : X(\omega)\cup\textrm{End}(X(\omega))\to \varphi (X(\omega))\cup \textrm{End}(\varphi (X(\omega))$.
	
	On the other hand, since $E(\omega)$ is a Cantor set in $[0, 1]$, the topology of $\textrm{End}(X(\omega))$ is the same as the Euclidean topology of $E(\omega)$.
	Hence, $\widehat \varphi$ is regarded as a homeomorphism on $\mathbb C$ which is quasiconformal on $\mathbb C\setminus E(\omega)$.
	Since $E(\omega)\subset [0, 1]$, it is removable with respect to quasiconformal mappings (cf. \cite{Lehto-Virtanen} \S 8), namely, $\widehat \varphi$ is a quasiconformal mapping on $\mathbb C$.
\end{proof}

Let $\varphi : X(\omega)\to X(\omega)$ be a $K$-quasiconformal homeomorphism.
 For the canonical pants decomposition $\{P_k^j\}_{k\in \mathbb Z_{\geq 0} ; j=1, 2, \dots , 2^k}$ and $d\in \mathbb N$, we put
 \begin{equation*}
 	W_d :=\textrm{Int}\left (\cup_{k=1}^{d-1} \cup_{j=1}^{2^k}(P_k^j\cup \gamma_{k+1}^{2j-1}\cup \gamma_{k+1}^{2j})\right ), \quad
 	V_d :=X(\omega)\setminus W_d.
 \end{equation*}
 Then, $X(\omega)=\cup_{d=1}^{\infty}W_d$, and the set $V_d$ contains all pants geodesics of depth $k$ for $k\geq d+1$.
 The relative boundary $\partial W_d$ of $W_d$ consists of $2^{d}$ pants geodesics of depth $d$.
 Since $\varphi (X(\omega))=X(\omega)$, there exists $D\in \mathbb N$ such that $\varphi (W_{k})\supset W_d$ for any $k\geq D$.
Therefore, for any pants geodesic $\gamma_k^j$ with depth $k\geq D+1$, $\varphi (\gamma_k^j)\subset V_d$.
We show a key lemma in our paper.

\begin{lemma}
\label{lemma: Key}
Let $\omega=(q_n)_{n=1}^{\infty}$ be in $\Omega$ with $\lim_{n\to\infty}q_n=1$.
	Then, for  any $K (\geq 1)$ there exists $d(K) \in \mathbb N$ depending on $K$ such that any $K$-quasiconformal mapping $\varphi  : X(\omega)\to X(\omega)$ sends any pants curve of depth $d\geq d(K)$ to a pants curve in $X(\omega)$.
\end{lemma}
\begin{proof}
For each $d\in \mathbb N$, we put
 \begin{equation*}
 	L(d):=\sup_{\gamma\subset V_d} \ell_{X(\omega)}(\gamma),
 \end{equation*}
 where the supremum is taken over all pants geodesics $\gamma$ in $V_d$.
 
 Since $\lim_{n\to\infty}q_n=1$, we have $\lim_{d\to\infty}L(d)=0$ from Proposition \ref{Pro:Length to 0}.
 The well-known collar lemma (cf. \cite{Buser}) tells us that there exists $W(\ell)\in \mathbb R_{>0}$ depending only on $\ell\in \mathbb R_{>0}$ such that every simple closed geodesic of length $\ell$ has a collar of width $W(\ell)$ and $W(\ell)\to +\infty$ as $\ell\to 0$.
 Therefore, $\delta (\omega):=\inf_{d\in \mathbb N} W(L(d)) >0$.
 So, we may find $d(K)\in \mathbb N$ such that $KL(d(K))<\delta (\omega)$.
 
 On the other hand, from Wolpert's lemma (Proposition \ref{Pro:Wolpert}), for a pants geodesic $\gamma$ of depth $d\geq d(K)$ we have
 \begin{equation*}
 	\ell_{X(\omega)}([\varphi (\gamma)])\leq K\ell_{X(\omega)}(\gamma)\leq KL(d)<\delta (\omega),
 \end{equation*}
 where $[\varphi (\gamma)]$ is the geodesic homotopic to $\gamma$.
 This means that $[\varphi (\gamma)]$ cannot pass through any collar of a pants geodesic.
 Therefore, we conclude that $\varphi (\gamma)$ has to be a pants curve if the depth is not less than $d(K)$. 
\end{proof}

From now on in this section, we assume that $\lim_{n\to \infty}q_n=1$ for $\omega=(q_n)_{n=1}^{\infty}$.

We consider the Teichm\"uller modular group $\textrm{Mod}(X(\omega))$ for $\omega\in\Omega$.
From Lemma \ref{lemma:QC-extension}, we have a homeomorphism $\psi |_{E(\omega)} : E(\omega)\to E(\omega)$ for any $\psi\in  [\varphi]\in \textrm{Mod}(X(\omega))$. The homeomorphism depends on the homotopy class $[\varphi]$ but it does not depend on the choice of $\psi\in [\varphi]$. 
Hence, we define a homeomorphism $[\varphi]|_{E(\omega)} : E(\omega)\to E(\omega)$ on $E(\omega)$  as $\psi|_{E(\omega)}$.
By using those homeomorphisms, we define some subgroups of $\textrm{Mod}(X(\omega))$ as follows.

We say that $[\varphi]\in \textrm{Mod}(X(\omega))$ is \emph{order-preserving} if $[\varphi]|_{E(\omega)}(a)<[\varphi]|_{E(\omega)}(b)$ whenever $a<b$ $(a, b\in E(\omega))$.
The set of order-preserving elements of $\textrm{Mod}(X(\omega))$ is denoted by $\textrm{Mod}(X(\omega))^{\textrm{OP}}$, and it is a subgroup of $\textrm{Mod}(X(\omega))$.

A M\"obius transformation $\Phi (z)=(z-i)/(z+i)$ sends the upper-half plane $\mathbb H$ onto the unit disk $\mathbb D$ and sends $\widehat R=\mathbb R\cup\{\infty\}$ to the unit circle $\Gamma$.
We say that three points $a, b, c$ on $\widehat{\mathbb R}$ are located in \emph{positive order} if $\Phi (a), \Phi (b), \Phi (c)$ are located along the positive direction on $\Gamma$.
We say that $[\varphi]\in \textrm{Mod}(X(\omega))$ is of \emph{positive order} if $[\varphi] (a), [\varphi] (b), [\varphi](c)$ are located in positive order whenever $a, b, c\in E(\omega)$ are located in positive order.
The set of positive order elements of $\textrm{Mod}(X(\omega))$ is denoted by $\textrm{Mod}(X(\omega))^{\textrm{PO}}$, and it is also a subgroup of $\textrm{Mod}(X(\omega))$.

For a pants curve $\gamma$, we put
\begin{equation*}
	E(\omega)_{\gamma} =\Delta (\gamma)\cap E(\omega),
\end{equation*}
where $\Delta (\gamma)$ is the Jordan domain bounded by $\gamma$.
A quasiconformal mapping $\varphi : X(\omega)\to X(\omega)$ is called \emph{piecewise order-preserving} if there exist pants geodesics $\gamma_1, \dots , \gamma_N$ such that
\begin{enumerate}
	\item $E(\omega)=\sqcup _{i=1}^{N} E(\omega)_{\gamma_i}$ ;
	\item  $\varphi |_{E(\omega)_{\gamma_i} }: E(\omega)_{\gamma_i}\to E(\omega)_{\varphi (\gamma_i)}$ is order-preserving.
\end{enumerate}

An element $[\varphi]\in \textrm{Mod}(X(\omega))$ is called \emph{piecewise order-preserving} if a representative $\varphi$ of $[\varphi]$ is piecewise order-preserving.
This definition is well-defined, namely, it does not depend on the choice of pants geodesics $\gamma_1, \dots , \gamma_N$.

The set of piecewise order-preserving elements in $\textrm{Mod}(X(\omega))$ is denoted by $\textrm{Mod}(X(\omega))^{\textrm{POP}}$. 
Let us confirm that $\textrm{Mod}(X(\omega))^{\textrm{POP}}$ is a subgroup of $\textrm{Mod}(X(\omega))$.

\begin{lemma}
	$\textrm{Mod}(X(\omega))^{\textrm{POP}}$ is a subgroup of $\textrm{Mod}(X(\omega))$.
\end{lemma}
\begin{proof}
It is obvious that $[id.]\in \textrm{Mod}(X(\omega))^{\textrm{POP}}$.

Let $[\varphi], [\psi]$ be in $\textrm{Mod}(X(\omega))^{\textrm{POP}}$.
Then, there exist pants geodesics $\gamma_1, \dots , \gamma_N$ (resp. $\delta_1, \dots , \delta_M$) given in the definition of $[\varphi]$ (resp. $[\psi]$).
We may assume that $\varphi$ and $\psi$ are $K$-quasiconformal mappings for some $K\geq 1$.
By replacing $\gamma_1, \dots , \gamma_N, \delta_1, \dots , \delta_M$ with deeper pants curves, we may assume that the depths of those curves are greater than $d(K)$, the quantity given in Lemma \ref{lemma: Key}.
Hence, $[\psi^{-1}(\gamma_j)]$ $(j=1, 2, \dots , M)$ are still pants geodesics because $K(\psi^{-1})=K(\psi)\geq K$.
Since $\psi (E(\omega))=E(\omega)$ and $E(\omega)=\sqcup _{i=1}^{N} E(\omega)_{\gamma_i}$, we have
$\psi (E(\omega))=\sqcup _{i=1}^{N}E(\omega)_{\psi^{-1}(\gamma_j)}$.

If $[\psi^{-1}(\gamma_j)]=\delta_i$ for some $i\in\{1, 2, \dots , M\}$, then $\varphi\circ\psi |_{E([\psi^{-1}(\gamma_j)])}=\varphi\circ\psi |_{E(\delta_j)} : E(\omega)_{[\psi^{-1}(\gamma_j)]}\to E(\omega)_{\varphi (\gamma_j)}$ is order-preserving.
We adopt the pants geodesic $[\psi^{-1}(\gamma_j)]$ as a pants geodesic for $\varphi\circ\psi$.

If $\Delta (\psi^{-1}( \gamma_j))$ contains some $\delta_i$, then we put
$P=\{i\in \{1, 2, \dots , M\}\mid \delta_i \subset \Delta (\psi^{-1}( \gamma_j))\}$.
Since $\psi |_{\delta_i}$ is order-preserving and $\psi (\delta_i)\subset \Delta (\gamma_j)$ for $i\in P$, we see that $\varphi\circ\psi |_{E_{\delta_i}} : E(\omega)_{\delta_i}\to E(\omega)_{\varphi\circ\psi (\delta_i)}$ is order-preserving if $i\in P$.
We adopt the pants geodesics $\delta_i$ $(i\in P)$ as pants geodesics for $\varphi\circ\psi$.

If $\psi^{-1}( \gamma_j)$ is contained in some $\Delta ( \delta_{i_0})$, then we put $Q=\{j\in \{1, 2, \dots , N\}\mid \psi^{-1}( \gamma_j)\subset \Delta ( \delta_{i_0})\} $.
For each $j\in Q$,  $\varphi\circ\psi |_{E([\psi^{-1}(\gamma_j)])} : E(\omega)_{[\psi^{-1}(\gamma_j)]}\to E(\omega)_{\varphi (\gamma_j)}$ is order-preserving.
We adopt the pants geodesics $\psi^{-1}(\gamma_j)$ $(j\in Q)$ as pants geodesics for $\varphi\circ\psi$.
The set of the adopted pants geodesics above shows that $\varphi\circ\psi$ is in $\textrm{Mod}(X(\omega))^{\textrm{POP}}$.

By considering deeper pants geodesics as above, we also see that $[\varphi]^{-1}\in \textrm{Mod}(X(\omega))^{\textrm{POP}}$ if $[\varphi]\in \textrm{Mod}(X(\omega))^{\textrm{POP}}$.
Thus, we have completed the proof of the lemma.
\end{proof}

\begin{lemma}
	$$
	\textrm{Mod}(X(\omega))^{\textrm{OP}}\subset \textrm{Mod}(X(\omega))^{\textrm{PO}}\subset \textrm{Mod}(X(\omega))^{\textrm{POP}}.
	$$
\end{lemma}
\begin{proof}
	It is obvious that $\textrm{Mod}(X(\omega))^{\textrm{OP}} \subset \textrm{Mod}(X(\omega))^{\textrm{PO}}$ and $\textrm{Mod}(X(\omega))^{\textrm{OP}} \subset \textrm{Mod}(X(\omega))^{\textrm{POP}}$.
	So, it suffices to show that $\textrm{Mod}(X(\omega))^{\textrm{PO}} \subset \textrm{Mod}(X(\omega))^{\textrm{POP}}$.
	
	Take $[\varphi]\in \textrm{Mod}(X(\omega))^{\textrm{PO}}$.
	If $\varphi (0)=0$, then $[\varphi]\in \textrm{Mod}(X(\omega))^{\textrm{OP}}$ and there is nothing to prove.
	We may assume that $x_0:=\varphi (0)\not=0$.
	
	Since $\varphi$ is a mapping of positive order, we see that $E(\omega)\cap (x_0, x_0+\varepsilon)\not=\emptyset$ for any $\varepsilon >0$.
	Hence, $x_0\not=1$.
	For the same reason, we have $x_1:=\varphi (1)\not=0$, $x_1<x_0$ and $E(\omega)\cap (x_1, x_0)=\emptyset$.
	We also see  for $y_0:=\varphi^{-1}(0), y_1:=\varphi^{-1}(1)$
that $0<y_1<y_0<1$ and $E(\omega)\cap (y_1, y_0)=\emptyset$.
So, we conclude that 
\begin{equation*}
	\varphi (E(\omega)\cap [0, y_1])=E(\omega)\cap [x_0, 1], \quad \varphi (E(\omega)\cap [y_0, 1])=E(\omega)\cap [0, x_1]
\end{equation*}
and $E(\omega)=(E(\omega)\cap [x_0, 1])\cup (E(\omega)\cap [0, x_1])$.
We may find pants curves $\gamma_1, \dots , \gamma_N, \gamma_{N+1}, \dots , \gamma_{N+M}$ so that
\begin{equation*}
	 E(\omega)\cap [0, x_1]=\cup_{i=1}^N E(\omega)_{\gamma_i}\quad \textrm{and} \quad  E(\omega)\cap [x_0, 1]=\cup_{i=N+1}^M E(\omega)_{\gamma_i}.
\end{equation*}
It is not hard to see that $\varphi|_{E(\omega)_{\gamma_i}}$ are order-preserving $(i=1, \dots , N+M)$.
Thus, we verify that $[\varphi]\in \textrm{Mod}(X(\omega))^{\textrm{POP}}$.
\end{proof}

\subsection{Pre-Teichm\"uller space pre-Teichm\"uller modular group}
Let $E(\omega)$ be a generalized Cantor set and $(X, f)$ a marked Riemann surface for $X(\omega)$.
From Remark \ref{Remark:Koebe} and Lemma \ref{lemma:QC-extension}, we may assume that $X\subset \mathbb C$ and $f$ is a quasiconformal mapping on $\mathbb C$ with $f(E(\omega))=\mathbb C\setminus X$.

We say that two marked Riemann surfaces $(X_i. f_i)$ $(i=1, 2)$ are \emph{pre-Teichm\"uller equivalent} if there exists a conformal mapping $F : \mathbb C\to \mathbb C$ such that $F\circ f_1|_{E(\omega)}=f_2|_{E(\omega)}$.
We denote by $\langle X, f\rangle$ the pre-Teichm\"uller equivalence class of $(X, f)$ and the set of the pre-Teichm\"uller equivalence classes of marked Riemann surfaces is called \emph{pre-Teichm\"uller space} of $X(\omega)$ and it is denoted by $\textrm{pre-Teich}(X(\omega))$.
We will see that $\textrm{pre-Teich}(X(\omega))$ is different from $\textrm{Teich}(X(\omega))$ (see Remark \ref{Rem:Different}). 

In the definition of the pre-Teichm\"uller equivalence, the conformal mappin $F$ is a M\"obius transformation (actually a linear map).
So, we may normalize that $f(0)=0$ and $f(1)=1$ in every marked Riemann surface $(X, f)$.

If two marked Riemann surfaces $(X_i, f_i)$ $(i=1, 2)$ are Teichm\"uller equivalent, then they are pre-Teichm\"uller equivalent.
Indeed, if $(X_1, f_1)$ and $(X_2, f_2)$ are Teichm\"uller equivalent, then there exists a M\"obius transformation $F : X_1\to X_2$ which is homotopic to $f_2\circ f_1^{-1}$. 
From our normalization, we see that $F(0)=f_2\circ f_1^{-1}(0)=0$ and $F(1)=f_2\circ f_1^{-1}(1)=1$.
Since $F$ is a linear mapping, we have $F=id.$ and $f_1|_{E(\omega)}=f_2|_{E(\omega)}$.

Hence, there exists a well-defined surjective map $\pi : \textrm{Teich}(X(\omega))\to \textrm{pre-Teich}(X(\omega))$ by
\begin{equation*}
	\pi ([X, f])=\langle X, f\rangle .
\end{equation*}

On $\textrm{pre-Teich}(X(\omega))$, we may define a distance $d_{pre-T}$ given by
\begin{equation*}
	d_{pre-T}(\langle X_1, f_1\rangle , \langle X_2, f_2\rangle )=\frac{1}{2}\inf_{\varphi}\log K(\varphi),
\end{equation*}
where the infimum is taken over all quasiconformal mappings $\varphi : \mathbb C\to \mathbb C$ with $\varphi \circ f_1=f_2$ on $E(\omega)$.
Because of the definition, we see
\begin{equation}
\label{eqn:pre-T-dist}
	d_{T}(P, Q)\geq d_{pre-T}(\pi (P), \pi (Q))
\end{equation}
for $P, Q\in \textrm{Teich}(X(\omega))$.

\begin{Pro}
	$d_{pre-T}$ is a complete distance on $\textrm{pre-Teich}(X(\omega))$.
\end{Pro}
\begin{proof}
	It is obvious that $d_{pre-T}$ satisfies the triangle inequality, $d_{pre-T}(P, P)=0$ and $d_{pre-T}(P, Q)=d_{pre-T}(Q, P)$ for $P, Q\in \textrm{pre-Teich}(X(\omega))$.
	By using a normal family argument, we see that if $d_{pre-T}(P, Q)=0$, then $P=Q$.
	So, $d_{pre-T}$ is a distance on $\textrm{pre-Teich}(X(\omega))$.
	We show the completeness of $d_{pre-T}$.
	
	Let $P_n=\langle X_n, f_n\rangle $ $(n=1, 2, \dots )$ be a Cauchy sequence in $\textrm{pre-Teich}(X(\omega))$ with respect to $d_{pre-T}$.
	For any $\varepsilon>0$, there exists $n_0\in \mathbb N$ such that 
	\begin{equation*}
		d_{pre-T}(P_n, P_{m})<\varepsilon
	\end{equation*}
	for any $n, m\geq n_0$.
	Hence, we may find quasiconformal mappings $f_{n_0}, f_n :\mathbb C\to \mathbb C$ such that $P_{n_0}=\langle X_{n_0}, f_{n_0}\rangle$, $P_n=\langle X_n, f_n\rangle$ and $K(f_n\circ f_{n_0}^{-1} )\leq e^{2\varepsilon}$ for $n\geq n_0$.
	Because of the normalization above, we see that $\{f_n\circ f_{n_0}^{-1} \}_{n=n_0}^{\infty}$ is a normal family of  $e^{2\varepsilon}$-quasiconformal mappings, and there exist a subsequence $\{f_{n_k}\}_{k=1}^{\infty}$ and a quasiconformal mapping $h : \mathbb C\to \mathbb C$ such that $\{f_{n_k}\circ f_{n_0}^{-1} \}_{k=1}^{\infty}$ converges to $h$ uniformly on $\widehat{\mathbb C}$ as $k\to \infty$, and $K(h)\leq e^{2\varepsilon}$ (\cite{Lehto-Virtanen} Theorem 5.4).
	For any $n\geq n_{n_0}$, we have
	\begin{eqnarray*}
		K((h\circ f_{n_0})\circ f_n^{-1})=K(h\circ (f_{n_0}\circ f_n^{-1})) \\
		\leq K(h)K(f_{n_0}\circ f_n^{-1})\leq e^{4\varepsilon}.
	\end{eqnarray*}
	This implies that for $P=\langle h\circ f_{n_0}(X(\omega)), h\circ f_{n_0}\rangle\in \textrm{pre-Teich}(X(\omega))$
	\begin{equation*}
		d_{pre-T}(P, P_n)\leq 2\varepsilon \quad (n\geq n_0),
	\end{equation*}
	and $P_n\to P$ as $n\to \infty$.
	It completes the proof of the proposition.
\end{proof}

Let $\varphi : X(\omega)\to X(\omega)$ be a quasiconformal mapping.
As we have already seen in Lemma \ref{lemma:QC-extension}, the mapping $\varphi$ is regarded as a quasiconformal mapping on $\mathbb C$.
Hence, a homomorphism $\Pi$ on $\textrm{Mod}(X(\omega))$ defined by
\begin{equation*}
	\Pi ([\varphi])=\varphi|_{E(\omega)}=:\langle [\varphi]\rangle
\end{equation*}
for $[\varphi]\in \textrm{Mod}(X(\omega))$ is well-defined.
We put $\textrm{pre-Mod}(X(\omega))=\Pi (\textrm{Mod}(X(\omega))$ and call it the \emph{pre-Teichm\"uller modular group}.
We also put
\begin{eqnarray*}
	\textrm{pre-Mod}(X(\omega))_F :=\Pi (\textrm{Mod}(X(\omega))^{OP}), \\\textrm{pre-Mod}(X(\omega))_T:=\Pi (\textrm{Mod}(X(\omega))^{PO}),
\end{eqnarray*}
and
\begin{equation*}
	\textrm{pre-Mod}(X(\omega))_V:=\Pi (\textrm{Mod}(X(\omega))^{POP}).
\end{equation*}

The pre-Teichm\"uller modular group $\textrm{pre-Mod}(X(\omega))$ acts on the pre-Teichm\"uller space $\textrm{pre-Teich}(X(\omega))$ by
\begin{equation*}
	\widetilde\chi_{\langle [\varphi]\rangle}(\langle X, f\rangle)=\langle X, f\circ\varphi^{-1}\rangle .
\end{equation*}
It is not hard to see that the action is well-defined and it gives a $d_{pre-T}$-isometry.

\subsection{Main results.}
Now, we exhibit our main results.

We say that a sequence $\omega=(q_n)_{n=1}^{\infty}\in \Omega$ satisfies the \emph{bounded rate divergence} (BRD) condition if $\lim_{n\to \infty}q_n=1$ and
\begin{equation*}
	\left |\log \frac{1-q_n}{1-q_{n+1}}\right |<M\quad (n=1, 2, \dots )
\end{equation*}
hold for some $M<\infty$.

\begin{Mythm}
\label{Thm:Exact}
	Suppose that $\omega\in \Omega$ satisfies the BRD condition.
	Then, there exists a homomorphism $\Theta$ which induces the following short exact sequences;
	\begin{eqnarray}
		0\to \textrm{Mod}(X(\omega))_0\xrightarrow{\iota} \textrm{Mod}(X(\omega))^{\textrm{OP}}\xrightarrow{\Theta}F\to 0,  \\
		0\to \textrm{Mod}(X(\omega))_0\xrightarrow{\iota} \textrm{Mod}(X(\omega))^{\textrm{PO}}\xrightarrow{\Theta}T\to 0, 
	\end{eqnarray}
	and
	\begin{equation}
		0\to \textrm{Mod}(X(\omega))_0\xrightarrow{\iota} \textrm{Mod}(X(\omega))^{\textrm{POP}}\xrightarrow{\Theta}V\to 0, 
	\end{equation}
	where $\textrm{Mod}(X(\omega))_0=\{[\varphi]\in \textrm{Mod}(X(\omega))\mid [\varphi]|_{E(\omega)}=id.\}$ and $\iota$ is the inclusion map.
\end{Mythm}

As we shall see in a later section, the construction of $\Theta$ is quite geometric.
Thus, we may find subgroups of $\textrm{Mod}(X(\omega))$ which are isomorphic to Thompson's groups $F$ and $T$.

Let $J(z)=\overline z$ be the complex conjugate.
We define a subgroup $\textrm{Mod}(X(\omega))_F$ by
\begin{eqnarray*}
	\textrm{Mod}(X(\omega))_F=\{[\varphi]\in \textrm{Mod}(X(\omega))^{\textrm{OP}} \mid \mbox{$J\circ \varphi$ is homotopic to $\varphi\circ J$}\}.
\end{eqnarray*}
\begin{Mythm}
\label{Thm:Isomorphism}
Suppose that $\omega\in \Omega$ satisfies the BRD condition.
	The homomorphism $\Theta$ in Theorem \ref{Thm:Exact} gives a group isomorphism between the groups $\textrm{Mod}(X(\omega))_F$ and $F$.
  Moreover, there exists a subgroup $\textrm{Mod}(X(\omega))_T$ of $\textrm{Mod}(X(\omega))$ such that $\textrm{Mod}(X(\omega))_F\subset \textrm{Mod}(X(\omega))_T$ and the homomorphism $\Theta$ in Theorem \ref{Thm:Exact} gives a group isomorphism between $\textrm{Mod}(X(\omega))_T$ and $T$.
\end{Mythm}
The exact construction of $\textrm{Mod}(X(\omega))_T$ is given in \S5.

\medskip

Next, we consider the actions of those groups on $(\textrm{Teich}(X(\omega)), d_T)$.
Here, we should notice that under the BRD condition, the action of the Teichm\"uller modular group on the Teichm\"uller space $\textrm{Teich}(X(\omega))$ is not properly discontinuous.
In fact, even the action of the small group $\textrm{Mod}(X(\omega))_0$ is not properly discontinuous. 

\begin{lemma}
	Suppose that $\omega\in \Omega$ satisfies the BRD condition.
	Then, the action of $\textrm{Mod}(X(\omega))_0$ on $(\textrm{Teich}(X(\omega)), d_T)$ is not properly discontinuous. 
\end{lemma}
\begin{proof}
	Let $\gamma_d$ be a pants geodesic of depth $d$ and $\tau_d$ the Dehn twist with respect to $\gamma_d$.
	Then, it is seen that $[\tau_d]\in \textrm{Mod}(X(\omega))_0$.
	Since $\ell_{X(\omega)}(\gamma_d)\to 0$ $(d\to +\infty)$ from Proposition \ref{Pro:Length to 0},  we see that $K([\tau_d])\to 1$ $(d\to +\infty)$, where $K([\tau_d]):=\inf \{K(\varphi) \mid \mbox{$\varphi$ is homotopic to $\tau_d$\}}$ (cf. \cite{Matsuzaki}).
	Hence, we have
	\begin{equation*}
		d_{T}([X(\omega), pd.], [X(\omega), \tau_{d}])\leq \frac{1}{2}
		\log K(\tau_d)\to 0 \quad (d\to \infty).
	\end{equation*}
	This implies that the action of $\textrm{Mod}(X(\omega))_0$ is not properly discontinuous.
\end{proof}

\begin{Rem}
\label{Rem:Different}
	The same idea as in the proof shows that $\textrm{Teich}(X(\omega))\not=\textrm{pre-Teich}(X(\omega))$.
	Indeed, there exists a sequence $\{k_d\}_{d\in \mathbb N}$ such that $\tau :=\prod_{d=1}^{\infty}\tau_d^{k_d} : X(\omega)\to X(\omega)$ is quasiconformal (cf. \cite{Matsuzaki}).
	On the other hand, we see that $\tau|{E(\omega)}=id.$.
	Hence, $[X(\omega), id,]\not=[X(\omega), \tau]$, but $\langle X(\omega), id.\rangle=\langle X(\omega), \tau \rangle$.
\end{Rem}

As for the actions of $\textrm{Mod}(X(\omega))_F$ and $\textrm{Mod}(X(\omega))_T$, which are isomorphic to $F$ and $T$ from Theorem \ref{Thm:Isomorphism}, respectively, they are properly discontinuous. 
\begin{Mythm}
\label{Thm:Discrete Action}
Suppose that $\omega\in \Omega$ satisfies the BRD condition.
	Then, $\textrm{Mod}(X(\omega))_F$ and $\textrm{Mod}(X(\omega))_T$ act properly discontinuously on $(\textrm{Teich}(X(\omega)), d_T)$.
\end{Mythm}

\begin{Rem}
	In \cite{FariaGardinerHarvey}, E. de Faria, F. P. Gardiner and W. J. Harvey showed the same statements as in Theorems \ref{Thm:Exact} and \ref{Thm:Discrete Action} for Thompson's group $F$ and the asymptotically Teichm\"uller space which is a quotient space of the Teichm\"uller space. 
\end{Rem}
Finally, we consider actions of the pre-Teichm\"uller modular groups on the pre-Teichm\"uller space, where we may consider the action of the Thompson's group $V$.
\begin{Mythm}
	\label{Thm:indiscrete}
	Suppose that $\omega\in \Omega$ satisfies the BRD condition.
	Then, $\textrm{pre-Mod}(X(\omega))_F, \textrm{pre-Mod}(X(\omega))_T$ and $\textrm{pre-Mod}(X(\omega))_V$ are isomorphic to $F, T$ and $V$, respectively.
	The actions of $\textrm{pre-Mod}(X(\omega))_F$ and $\textrm{pre-Mod}(X(\omega))_T$ on $(\textrm{pre-Teich}(X(\omega)), d_{pre-T})$ are properly discontinuous, but the action of $\textrm{pre-Mod}(X(\omega))_V$ is not. 
\end{Mythm}

\section{Quasiconformal mappings on generalized Cantor sets}
In this section, we give some results obtained from arguments in our previous papers \cite{Shigadynamical}, \cite{ShigaStructure}. 

Let $\omega=(q_n)_{n=1}^{\infty}\in \Omega$ with $\inf _{n\in \mathbb N}q_n=\delta>0$.
We look at the Euclidean structure of $X(\omega)$.

Let $I_k^i$ $(k\in \mathbb N, i=1, 2, \dots , 2^k)$ be closed intervals given in the $k$-th step in the construction of $E(\omega)$ and $J_k^i$ the open interval between $I_k^i$ and $I_k^{i+1}$.
We have
\begin{equation}
\label{eqn:Length Of I}
	|I_k^i|=2^{-k}\prod_{j=1}^k (1-q_j)
\end{equation}
and
\begin{equation}
\label{eqn:Length Of J}
	|J_k^i|\geq 2\delta |I_k^i|=2\delta|I_k^1|.
\end{equation}
Let $C_k^i$ be a circle of radius $r(k):= \frac{1}{2}(1+\delta)|I_k^1|$ centered at $c(k;i)$, the midpoint of $I_k^i$.
From (\ref{eqn:Length Of I}) and (\ref{eqn:Length Of J}), we see that
\begin{equation*}
	C_k^{i}\cap C_k^{i'}=\emptyset \ (i\not=i')\ \mbox{and}\ C_k^i\cap C_{k+1}^j=\emptyset.
\end{equation*}
Hence, each $C_k^i$ is a pants curve, and $C_k^i, C_{k+1}^{2i-1}, C_{k+1}^{2i}$ $(i=1, 2, \dots , 2^k)$ bound a pair of pants $Q_k^i$.
By imitating the method of Step 5 in \S 4 of \cite{Shigadynamical}, we have the following.
\begin{lemma}
\label{lemma : Circle Stretching}
	Suppose that there exists $M<\infty$ such that
	\begin{equation*}
		\left | \log \frac{1-q_k}{1-q_{k+1}}\right |, \left | \log \frac{1-q_{k+1}}{1-q_{k+2}}\right |<M.
	\end{equation*}
	Then, there exists a $K$-quasiconformal mapping $\varphi_{k, i} : Q_k^i\to Q_{k+1}^j$ $(i\in \{1, 2, \dots , 2^k\};  j\in \{1, 2, \dots , 2^{k+1}\})$ with $\varphi _{k, i}(C_k^i)=C_{k+1}^j, \varphi_{k, i} (C_{k+1}^{2i-1})=C_{k+2}^{2j-1}$ and $\varphi_{k, i} (C_{k+1}^{2i})=C_{k+1}^{2j}$.
	Furthermore, the maximal dilatation $K$ depends only on $\delta$ and $M$, and for any $\theta\in [0, 2\pi)$
	\begin{eqnarray}
		\varphi_{k, i} (c(k;i)+r(k)e^{i\theta})&=&c(k+1;j)+r(k+1)e^{i\theta} \\
		\varphi _{k, i}(c(k+1;2i-1)&+&r(k+1)e^{i\theta}) \\ &=&c(k+2;2j-1)+r(k+2)e^{i\theta}\nonumber \\
				\varphi_{k, i} (c(k+1;2i)+r(k+1)e^{i\theta})&=&c(k+2;2j)+r(k+2)e^{i\theta}
	\end{eqnarray}
	hold.
\end{lemma}

 \begin{Rem}
 \label{Rem:Angle Keep}
 	The equations (3.3)-(3.4) show that the  mapping $\varphi_{k, i}$ is keeping angles on $\partial Q_{k}^{i}$.
 	That is, for each $z\in C_{k}^{i}$, $\arg (z-c(k;i))=\arg (\varphi_{k, i}(z)-c(k+1;j))$, and for each $z\in C_{k+1}^{2i-1}$ (resp. $z\in C_{k+1}^{2i}$), $\arg (z-c_{k+1, 2i-1})=\arg (\varphi_{k, i}(z)-c(k+2, 2i-1))$ (resp. $\arg (z-c(k+1; 2i))=\arg (\varphi_{k,i}(z)-c(k+2;2j))$).
 \end{Rem}

\section{Proof of Theorem \ref{Thm:Exact}}
Let $X(\omega)=\mathbb C\setminus E(\omega)$ for $\omega\in\Omega$ satisfying the BRD condition, and let
 $\varphi : X(\omega)\to X(\omega)$ be a $K$-quasiconformal homeomorphism.
 For the canonical pants decomposition $\{P_k^j\}_{k\in \mathbb Z_{\geq 0} ; j=1, 2, \dots , 2^k}$ and $d\in \mathbb N$, we put
 \begin{equation*}
 	W_d :=\textrm{Int}\left (\cup_{k=1}^{d-1} \cup_{j=1}^{2^k}(P_k^j\cup \gamma_{k+1}^{2j-1}\cup \gamma_{k+1}^{2j})\right ), \quad
 	V_d :=X(\omega)\setminus W_d,
 \end{equation*}
 as in \S 2.4.
 Then, $X(\omega)=\cup_{d=1}^{\infty}W_d$, and the set $V_d$ contains all pants geodesics of depth $k$ for $k\geq d+1$.
 The relative boundary $\partial W_d$ of $W_d$ consists of $2^{d}$ pants geodesics of depth $d$.
 Since $\varphi (X(\omega))=X(\omega)$, there exists $D\in \mathbb N$ such that $\varphi (W_{k})\supset W_d$ for any $k\geq D$.
Therefore, for any pants geodesic $\gamma_k^j$ with depth $k\geq D+1$, $\varphi (\gamma_k^j)\subset V_d$.

 \medskip
 
 \noindent
 {\bf The construction of the homomorphism $\Theta$.}
  First, we construct $\Theta ([\varphi])$ for $[\varphi]\in \textrm{Mod}(X(\omega))^{OP}$.
 Let $\varphi$ be a $K$-quasiconformal mapping.
 From Lemma \ref{lemma: Key}, $\varphi (\gamma_d^j)$ is a pants curve if $d\geq d(K)$ and $j\in \{1, 2, \dots , 2^d\}$.
 
 Let $Y_{\varphi}$ be an unbounded domain with $\partial Y_{\varphi}=[\varphi ( \gamma_{d(K)}^1)]\cup [\varphi ( \gamma_{d(K)}^2)]\cup \dots \cup [\varphi(\gamma_{d(K)}^{2^{d(K)}})]$, where $[\varphi (\gamma_{d(K)}^j)]$ is the hyperbolic geodesic homotopic to $\varphi (\gamma_{d(K)}^j)$ $(j\in \{1, 2, \dots , 2^{d(K)}\} )$.
 Then, both $W_{d(K)}$ and $Y_{\varphi}$ are unbounded domains bounded by pants geodesics and they do not depend on the choice of a representative in the homotopy class of $[\varphi]$.
 It follows from Proposition \ref{Pro:Ordered rooted binary tree} that $D_{[\varphi]}^{d(K)}=\iota_C(T_{W_{d(K)}})$ and $R_{[\varphi]}^{d(K)}:=\iota_C(T_{Y_{[\varphi]}})$ are ordered rooted binary trees with the same number of leaves.
 We put a number $j$ to the leaves $\iota_C (\varphi (\gamma_{d(K)}^j))$ and $\iota_C(\gamma_{d(K)}^j)$ for each $j$ $(j\in \{1, 2, \dots , 2^{d(K)}\})$.
 Since $[\varphi]$ is in $\textrm{Mod}(X(\omega))^{\textrm{OP}}$, the tree diagram $(D_{[\varphi]}^{d(K)}, R_{[\varphi]}^{d(K)})$ is order-preserving and it gives an element of $F$ from Proposition \ref{Pro:Tree for F}.
 
 Furthermore, the element does not depend on the choice of the depth $d\geq d(K)$.
 Namely, we have the same element as $(D_{[\varphi]}^{d(K)}, R_{[\varphi]}^{d(K)})$ from $(D_{[\varphi]}^d, R_{[\varphi]}^d)$ for $d> d(K)$.
 Indeed, we may take $d(K)+1$ instead of $d(K)$. Then  $\gamma_{d(K)+1}^{2j-1}, \gamma_{d(K)+1}^{2j}\in \partial W_{d(K)+1}$ and $\gamma_{d(K)}^j\in \partial W_{d(K)}$ are boundary curves of  a pair of pants $P_{d(K)}^j$.  
  Also, $\varphi (\gamma_{d(k)+1}^{2j-1}), \varphi (\gamma_{d(K)+1}^{2j})$ and $\varphi (\gamma_{d(K)}^j)$ are pants curves.
  Therefore, we see that the tree diagram $(D_{[\varphi]}^{d(K)+1}, R_{[\varphi]}^{d(K)+1})$ is obtained from $(D_{[\varphi]}^{d(K)}, R_{[\varphi]}^{d(K)})$ by applying caret operations given in \S2.1. 
  Hence, it gives the same element in $F$ as $(D_{[\varphi]}^{d(K)}, R_{[\varphi]}^{d(K)})$ does, which is denoted by $\Theta ([\varphi])$.
  Thus, we construct a well-defined map $\Theta : \textrm{Mod}(X(\omega))^{\textrm{OP}}\to F$.
  In fact, we may see that for any domain $W\supset W_{d(K)}$ bounded by pants curves with $W\cap E(\omega)=\emptyset$, the tree diagram $(\iota_C (T_W), \iota_C (T_{\varphi (W)}))$ determines $\Theta ([\varphi])$ because the tree diagram is also obtained by applying caret operations to $(D_{[\varphi]}^{d(K)}, R_{[\varphi]}^{d(K)})$.
  
  Next, we show that $\Theta$ is a homomorphism.
   Take $[\varphi_1], [\varphi_2]\in \textrm{Mod}(X(\omega))^{\textrm{OP}}$.
As we have constructed above, there exist order-preserving tree diagrams $(D_{[\varphi_i]}^{d_i}, R_{[\varphi_i]}^{d_i})$ for $[\varphi_i]$ $(i=1, 2)$.
Take $d\in \mathbb N$ so large that $W_d\supset W_{d_2}$ and $\varphi_2 (W_d)\supset W_{d_1}$, where $d_i=d(K(\varphi_i))$ $(i=1, 2)$.
Then, we obtain an order-preserving tree diagram $(D_{\varphi_1\circ \varphi_2}^d, R_{\varphi_1\circ \varphi_2}^d)$ for $\varphi_1\circ \varphi_2$ which determines $\Theta ([\varphi_1]\circ [\varphi_2])$.
On the other hand, we have $\Theta ([\varphi_2])$ from $\varphi_2|_{W_d} : W_d\to \varphi_2 (W_d)$.
Moreover, since $W_{d_1}\subset \varphi_2 (W_d)$, we have $\Theta ([\varphi_1])$ from $\varphi_1|_{\varphi_2 (W_d)} : \varphi_2 (W_d) \to \varphi_1\circ\varphi_2 (W_d)$.
Hence, we verify that $\Theta ([\varphi_1\circ\varphi_2])=\Theta ([\varphi_1])\circ\Theta([\varphi_2])$ and $\Theta$ is a homomorphism.

We show that $\Theta$ is surjective on $\textrm{Mod}(X(\omega))^{\textrm{OP}}$.
Let $\Omega_0^1$ be an unbounded domain bounded by $C_1^1, C_2^3$ and $C_2^4$, and $\Omega_0^2$ an unbounded domain bounded by $C_2^1, C_2^2$ and $C_1^2$, where $C_k^i$ $(k, i\in \mathbb N)$ are circles given in \S3.
There exists a quasiconformal mapping $\phi_0 : \Omega_0^1\to \Omega_0^2$ with $\phi_0 (C_1^1)=C_2^1, \phi_0 (C_2^3)=C_2^2$ and $\phi_0 (C_2^4)=C_1^2$ (Figure \ref{Fig.Surjective}). 
\begin{figure}[htbp]
\center
\includegraphics[width=12cm]{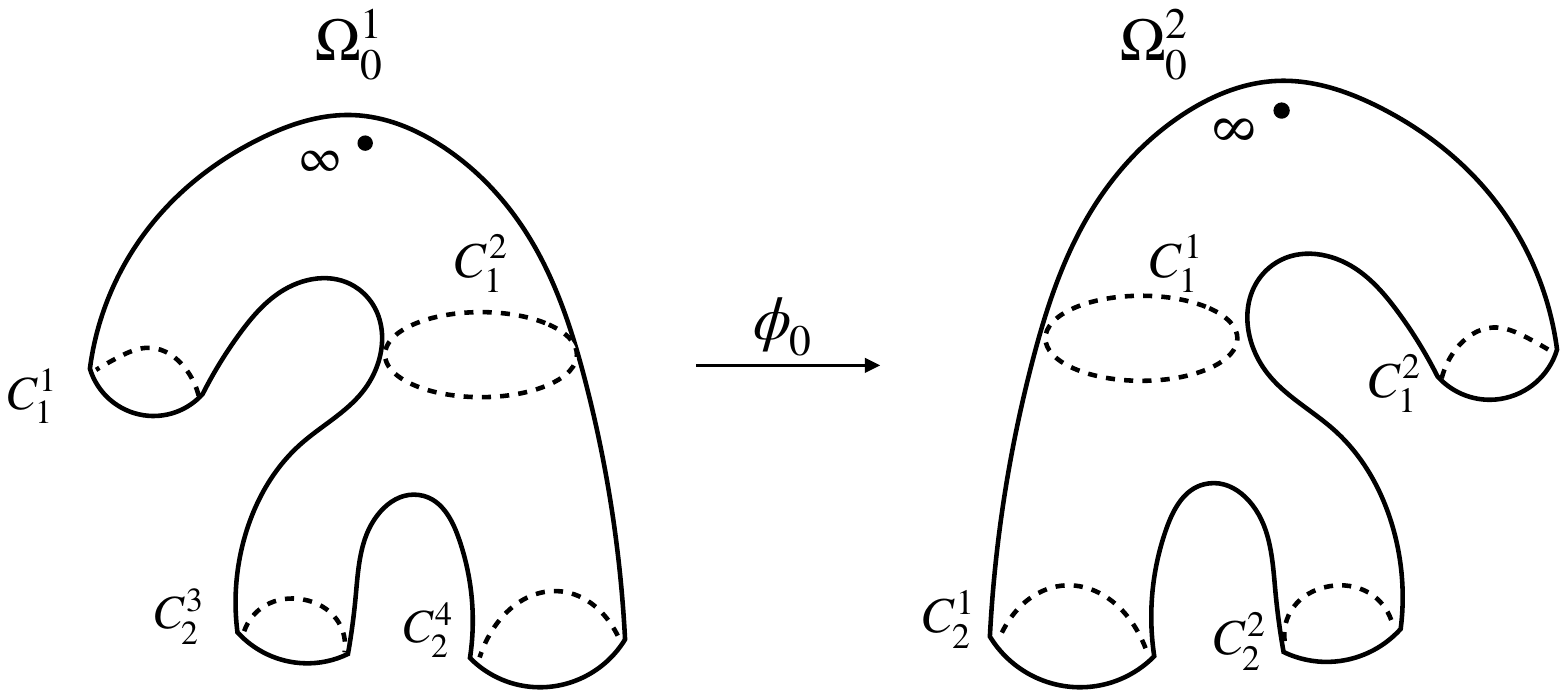}
\label{Fig.Surjective}
\caption{ }
\end{figure}
Moreover, as in Lemma \ref{lemma : Circle Stretching}, we may assume that for each $\theta\in [0, 2\pi)$,
\begin{eqnarray*}
	\phi_0 (c(1,1)+r(1)e^{i\theta})=c(2, 1)+r(2)e^{i\theta}, \\
	\phi_0 (c(2,3)+r(2)e^{i\theta})=c(2, 2)+r(2)e^{i\theta},
\end{eqnarray*}
and
\begin{equation*}
	\phi_0 (c(2,4)+r(2)e^{i\theta})=c(1, 2)+r(1)e^{i\theta}
\end{equation*}
hold.
We extend the quasiconformal mapping $\phi_0$ on $\Omega_0^1$ to a quasiconformal mapping on $X(\omega)$.

From Lemma \ref{lemma : Circle Stretching}, there exists a $K(\delta, M)$-quasiconformal mapping $\varphi_{1, 1} : Q_1^1\to Q_2^1$ satisfying
\begin{equation*}
	\varphi_{1,1}(c(1,1)+r(1)e^{i\theta})=c(2, 1)+r(2)e^{i\theta}
\end{equation*}
on $C_1^1$, where $K(\delta, M)(\geq 1)$ is a constant depending only on $\delta$ and $M$.
The mapping $\varphi_{1,1}$ agrees with $\phi_0$ on $C_1^1$.
Hence, the quasiconformal mapping $\phi_0$ extends to a $K$-quasiconformal mapping from $\Omega_0^1\cup Q_1^1$ onto $\Omega_0^2\cup Q_2^1$, where $K=\max \{K(\phi_0), K(\delta, M)\}$ depending only on $\phi_0, \delta$ and $M$.
Repeating this process, we may construct a $K$-quasiconformal mapping $\Phi_0 : X(\omega)\to X(\omega)$.
Since $\phi_0$ keeps the order of boundary curves from left to right, the extended mapping $\Phi_0$ is order-preserving.
 Therefore, we see that $[\Phi_0]\in \textrm{Mod}(X(\omega))^{\textrm{OP}}$.
From the construction of $\Theta$, we verify that $\Theta ([\Phi_0])=f_0\in F$ given in Example 2.1.
\begin{figure}[htbp]
	\center
	\includegraphics[width=12cm]{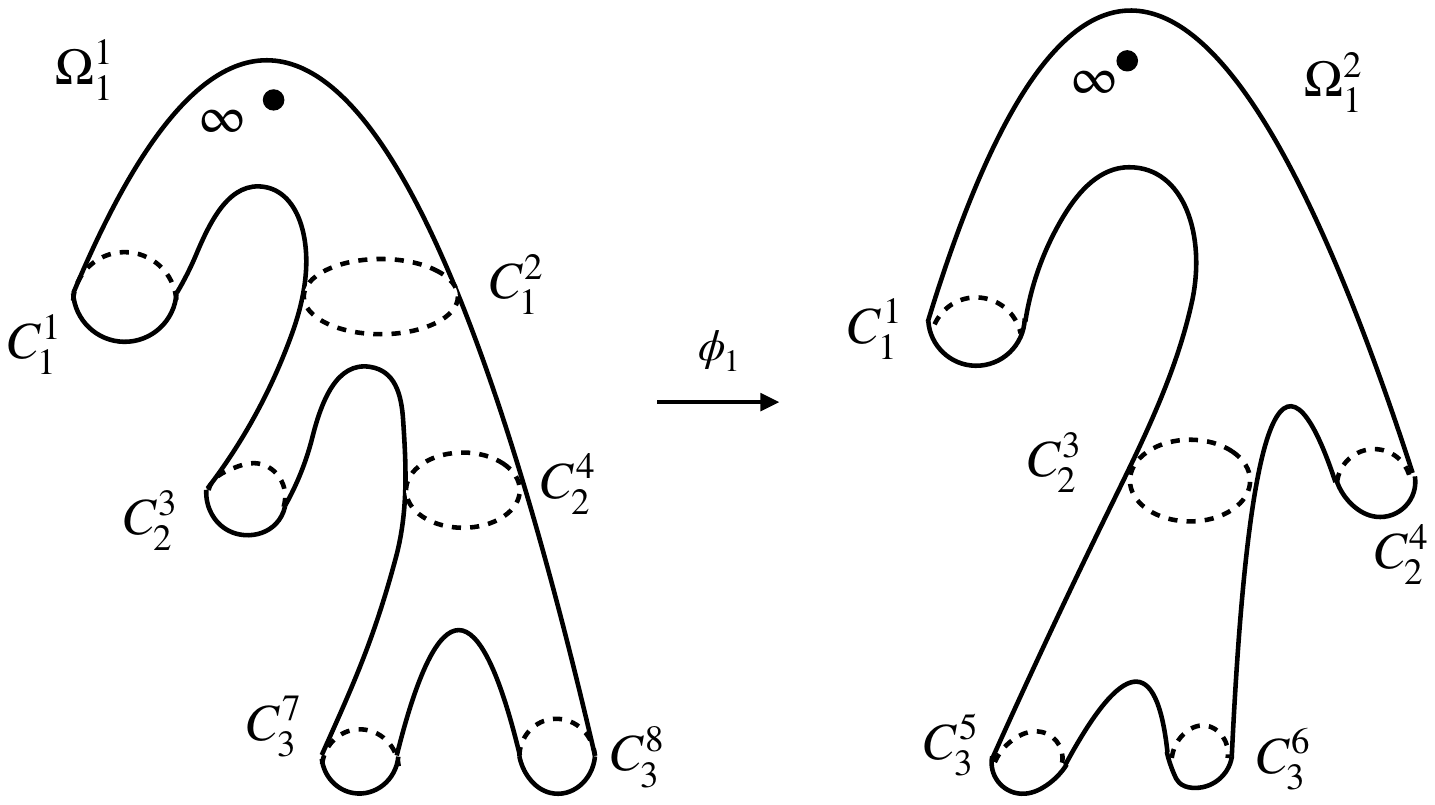}
		\caption{ }
	\label{Fig.Surjective2}
\end{figure}

As for $f_1$, we consider domains $\Omega_1^1, \Omega_1^2$ and a quasiconformal mapping $\phi_1 : \Omega_1^1 \to \Omega_1^2$ as in Figure \ref{Fig.Surjective2}.
By using the same argument as above, we may show that the quasiconformal mapping $\phi_1$ extends to an order-preserving quasiconformal mapping $\Phi_1$ with $\Theta ([\Phi_1])=f_1$.

Since $\Theta$ is a homomorphism and the functions $f_0, f_1$ generate $F$ (Proposition \ref{Pro:Generator F}), we verify that $\Theta (\textrm{Mod}(X(\omega))^{\textrm{OP}})=F$ as desired.

\begin{Rem}
\label{Rem:NoTwist}
	In extending $\phi_i$ to $\Phi_i$ $(i=0, 1)$, we use mappings $\varphi_{k, \ell}$ in Lamma \ref{lemma : Circle Stretching}. 
	As we have noted in Remark \ref{Rem:Angle Keep}, those mappings keep angles on the boundaries.
	Therefore, the mappings $\Phi_i$ $(i=0, 1)$ do not have Dehn twists with respect to pants curve of $\partial \Omega_i^1$ and any pants curve deeper than them.
\end{Rem}

Next, we consider the kernel $\textrm{Ker}(\Theta)$ of $\Theta$.
If $[\varphi]\in \textrm{Mod}(X(\omega))^{\textrm{OP}}$ belongs to $\textrm{Ker}(\Theta)$, then $\Theta ([\varphi])=id.\in F$.
This means that for a sufficiently large $k\in \mathbb N$, $D_{\varphi}^k=R_{\varphi}^k$.
Hence, we verify that $\varphi|_{E(\omega)}=id.$
Conversely, if $\Theta ([\varphi])\not=id.$, then for a sufficiently large $k\in \mathbb N$ there exists a component $\gamma$ of $\partial W_k$ such that the depth of $\varphi (\gamma)$ is not the same as the depth of $\gamma$.
Hence, we see that $E(\omega)_\gamma\not=E(\omega)_{\varphi (\gamma)}=\varphi (E(\omega)_\gamma)$ and $\varphi |_{E(\gamma)}\not=id.$
Hence, we verify that $\textrm{Ker}(\Theta)=\textrm{Mod}(X(\omega)_0)$ and we obtain the exact sequence
\begin{equation*}
	0\to \textrm{Mod}(X(\omega))_0\xrightarrow{\iota} \textrm{Mod}(X(\omega))^{\textrm{OP}}\xrightarrow{\Theta}F\to 0.
\end{equation*}

Now, we define $\Theta ([\varphi])$ for $[\varphi]$ in $\textrm{Mod}(X(\omega))^{\textrm{POP}}$, and in  $\textrm{Mod}(X(\omega))^{\textrm{PO}}$.
Suppose that $\varphi : X(\omega)\to X(\omega)$ is a $K$-quasiconformal mapping.
From the definition of piecewise order-preserving maps, there exist pants curves $\gamma_1, \dots , \gamma_N$ such that $E(\omega)=\sqcup_{i=1}^N E(\omega)_{\gamma_i}$ and $\varphi |_{E(\omega)_{\gamma_i}} : E(\omega)_{\gamma_i}\to E(\omega)_{\varphi ( \gamma_i)}$ is order-preserving for each $i\in \{1, \dots , N\}$.

Let $\Omega_{\varphi}\subset\mathbb C$ be the unbounded domain bounded by $\gamma_1, \dots , \gamma_N$.
We may take $\Omega_{\varphi}$ so that $\Omega_{\varphi}\supset W_{d(K)}$.
Then, $\iota_{C}(T_{\Omega_{\varphi}}), \iota_C (T_{\varphi (\Omega_{\varphi})})$ are ordered rooted trees and by giving the correspondence of their leaves via $\varphi$, we see that the tree diagram $(\iota_{C}(T_{\Omega_{\varphi}}), \iota_C (T_{\varphi (\Omega_{\varphi})}))$ determines an element in $V$, which is denoted by $\Theta ([\varphi])$.

Moreover, we see that $\Theta ([\varphi])$ is well-defined.
Indeed, since $\varphi |_{E(\omega)_{\gamma_i}} : E(\omega)_{\gamma_i}\to E(\omega)_{\varphi ( \gamma_i)}$ is order-preserving for each $i\in \{1, \dots , N\}$, for any domain $\Omega\supset \Omega_{\varphi}$ bounded by pants curves with $\overline \Omega\cap E(\omega)=\emptyset$, the tree diagram $(\iota_C (T_{\Omega}), \iota_C (T_{\varphi (\Omega)}))$ determines $\Theta ([\varphi])\in V$.
We also see that the mappings $\Theta : \textrm{Mod}(X(\omega))^{\textrm{PO}}\to T$ and $\Theta : \textrm{Mod}(X(\omega))^{\textrm{POP}}\to V$ are well-defined homeomorphisms by using the same argument as in $\textrm{Mod}(X(\omega))^{OP}$.

To show the surjectivity, it suffices to show that there exist quasiconformal mappings $\Phi_2, \Phi_3 : X(\omega)\to X(\omega)$ such that $\Theta ([\Phi_2])=f_2$ and $\Theta ([\Phi_3])=f_3$. 
\begin{figure}[htbp]
	\center
	\includegraphics[width=12cm]{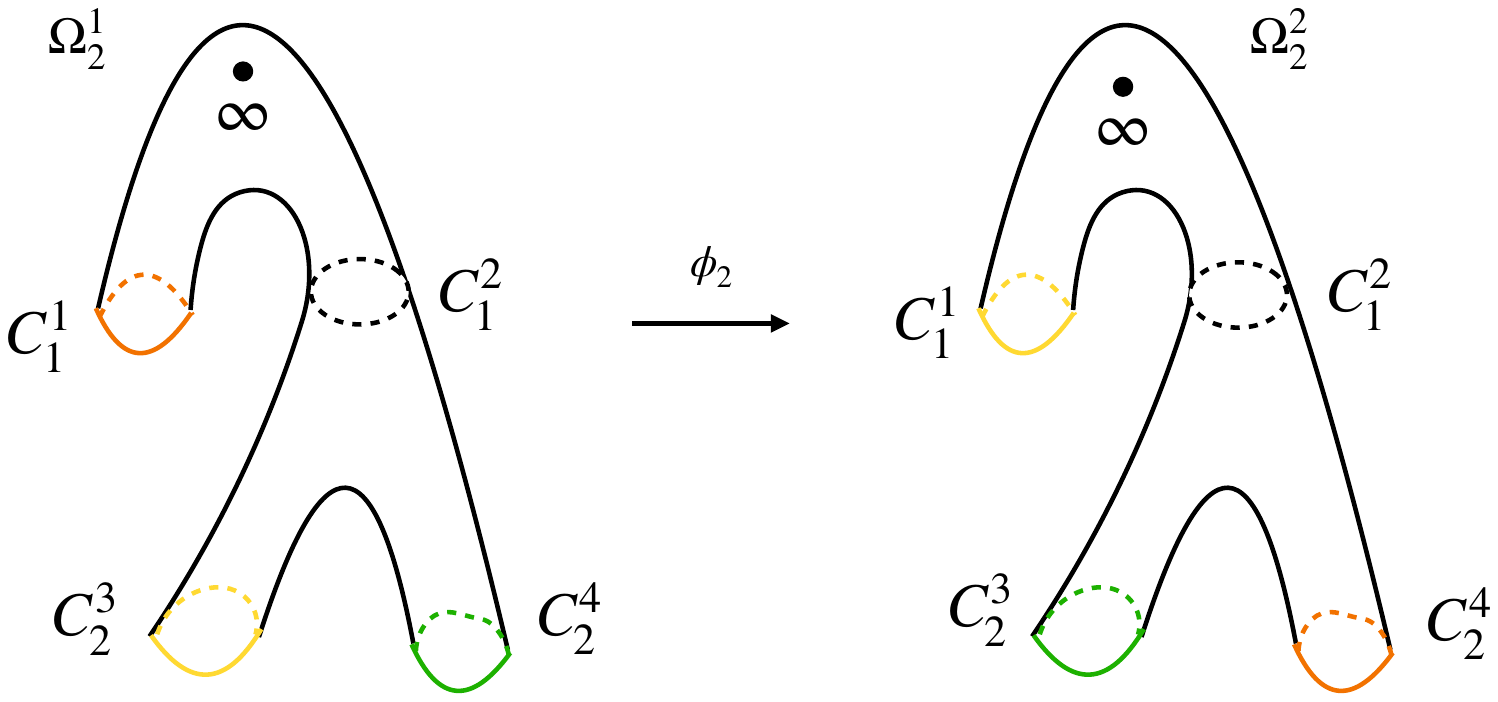}
		\caption{ }
	\label{Fig.SurjectiveT}
\end{figure}

To obtain $\Phi_2$, we consider a quasiconformal mapping $\phi_2 : \Omega_2^1\to \Omega_2^2$ with $\phi_2 (C_1^1)=C_2^4, \phi_2 (C_2^3)=C_1^1$ and $\phi_2 (C_2^4)=C_2^3$ as in Figure \ref{Fig.SurjectiveT}.
By using the same argument as in the case of $\phi_0$, the quasiconformal mapping $\phi_2$ extends to a quasiconformal self-mapping $\Phi_2$ of $X(\omega)$.
Since $\phi_2$ exchanges $C_1^1, C_2^3, C_2^4$ cyclically, the extended mapping $\Phi_2$ is of positive order.
Because of the construction of $\Theta$, we see that $\Theta ([\Phi_2])=f_2$.
Thus, we have $\Theta (\textrm{Mod}(X(\omega))^{\textrm{PO}})=T$ from Proposition \ref{Pro:Generator T}.
\begin{figure}[htbp]
	\center
	\includegraphics[width=12cm]{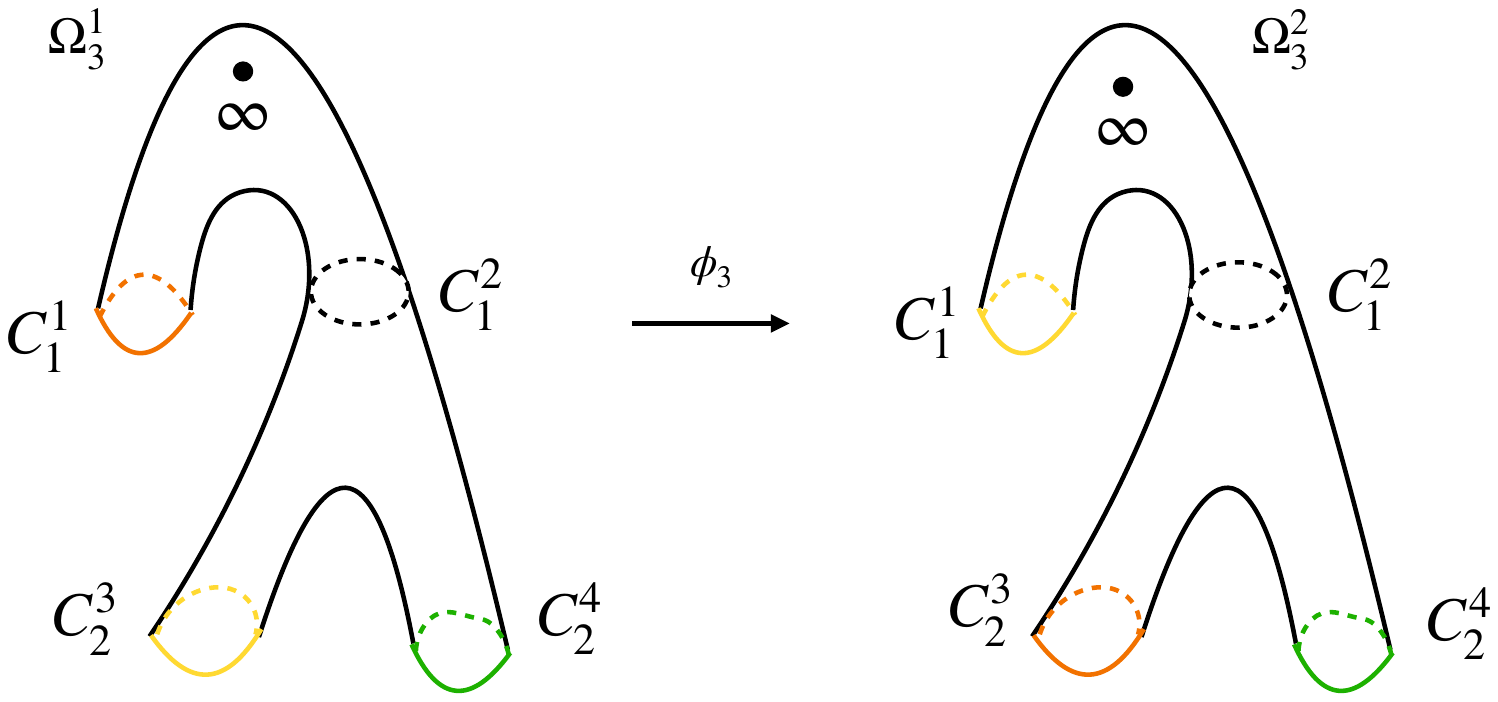}
		\caption{ }
	\label{Fig.SurjectiveV}
\end{figure}

To obtain $\Phi_3$, we consider a quasiconformal mapping $\phi_3 : \Omega_3^1\to \Omega_3^2$ with $\phi_3 (C_1^1)=C_2^3, \phi_3 (C_2^3)=C_1^1$ and $\phi_3 (C_2^4)=C_2^4$ as in Figure \ref{Fig.SurjectiveV}.
By using the same argument as in the case of $\phi_0$, the quasiconformal mapping $\phi_3$ extends to a quasiconformal self-mapping $\Phi_3$ of $X(\omega)$.
Since $\phi_3$ exchanges $C_1^1$ and $C_2^3$, the extended mapping $\Phi_3$ is piecewise order-preserving.
Because of the construction of $\Theta$, we see that $\Theta ([\Phi_3])=f_3$.
Thus, we have $\Theta (\textrm{Mod}(X(\omega))^{\textrm{POP}})=V$ from Proposition \ref{Pro:Generator V}.

\begin{Rem}
\label{Rem:NoTwists2}
	By the same reason as in Remark \ref{Rem:NoTwist}, we see that the mappings $\Phi_i$ $(i=2, 3)$ do not have Dehn twists with respect to $C_1^1, C_2^3$ and $C_2^4$ and pants curves deeper than them.
\end{Rem}

Since we may see that the kernel of $\Theta$ is still $\textrm{Mod}(X(\omega))_0$, we obtain exact sequences
\begin{eqnarray*}
	0\to \textrm{Mod}(X(\omega))_0\xrightarrow{\iota} \textrm{Mod}(X(\omega))^{\textrm{PO}}\xrightarrow{\Theta}T\to 0, \\
	0\to \textrm{Mod}(X(\omega))_0\xrightarrow{\iota} \textrm{Mod}(X(\omega))^{\textrm{POP}}\xrightarrow{\Theta}V\to 0.
\end{eqnarray*}
We complete the proof of Theorem \ref{Thm:Exact}.

\section{Proof of Theorem \ref{Thm:Isomorphism}}
\noindent
{\bf $\textrm{Mod}(X(\omega))_F$. }
Let $[\Phi_i] $ $(i=0, 1)$ be elements given in \S 4.
Since there exist $\varphi_i\in [\Phi_i]$ $(i=0, 1)$ such that $\varphi_i : X(\omega)\to X(\omega)$ are symmetric with respect to the real line $\mathbb R$, we see that $[\Phi_0], [\Phi_1]\in \textrm{Mod}(X(\omega))_F$.
Hence, we have 
$\Theta (\textrm{Mod}(X(\omega))_F)=F$.
We show that $\Theta$ is injective on $\textrm{Mod}(X(\omega))_F$.

First, we observe the action of $\textrm{Mod}(X(\omega))_F$ on simple closed geodesics which are symmetric with respect to $\mathbb R$.

\begin{lemma}
\label{lemma:symmetric}
	Let $\gamma$ be a simple closed geodesic with $J(\gamma)=\gamma$.
	Then, $J([\varphi (\gamma)])=[\varphi (\gamma)]$ for any $[\varphi]\in \textrm{Mod}(X(\omega))_F$.
	That is, the geodesic $[\varphi (\gamma)]$ is still symmetric with respect to $\mathbb R$.
\end{lemma}
\begin{proof}
	Since $\varphi\circ J$ is homotopic to $J\circ \varphi$, $\varphi\circ J(\gamma)=\varphi (\gamma)$ is homotopic to $J\circ \varphi (\gamma)$.
	This implies that $[\varphi (\gamma)]=[J\circ \varphi (\gamma)]$.
	Hence, we have
	\begin{equation*}
		J([\varphi (\gamma)]=J([J\circ \varphi (\gamma)]).
	\end{equation*}
	On the other hand, $J : X(\omega)\to X(\omega)$ is a hyperbolic isometry. Therefore, we see that
$J([c])=[J(c)]$ for any closed curve $c$ in $X(\omega)$.
Thus, we obtain
\begin{equation*}
	J([\varphi (\gamma)]=[J^2 (\varphi (\gamma))]=[\varphi (\gamma)]
\end{equation*}
as desired.
\end{proof}

Let $[\varphi]$ be in $\textrm{Mod}(X(\omega))_F$ with $\Theta ([\varphi])=id.$
Then, $\varphi|_{E(\omega)}=id.$ and $\varphi (\gamma_d^j)$ are homotopic to $\gamma_d^j$ $(j=1, 2, \dots , 2^d)$ for a sufficiently large $d\in \mathbb N$.
We show that $\varphi (\gamma_{d-1}^j)$ is homotopic to $\gamma_{d-1}^j$ $(j\in \{1, 2, \dots , 2^{d-1}\})$.

We put $\gamma_{\varphi, d-1}^j:=[\varphi (\gamma_{d-1}^j)]$.
We show that $\gamma_{\varphi, d-1}^j=[\gamma_{d-1}^j]$.

Since $[J\circ \varphi]=[\varphi\circ J]$ and $J(\gamma_{d-1}^j)=\gamma_{d-1}^j$, we have
\begin{equation*}
	\gamma_{\varphi, d-1}^j=[\varphi (\gamma_{d-1}^j)]=[\varphi\circ J(\gamma_{d-1}^j)]=[J\circ\varphi (\gamma_{d-1}^j)].
\end{equation*}
Hence, $J(\gamma_{\varphi, d-1}^j)$ is homotopic to $J\circ J(\varphi (\gamma_{d-1}^j))=\varphi (\gamma_{d-1}^j)$.
However, $J$ is a hyperbolic isometry on $X(\omega)$. 
Therefore, $J(\gamma_{\varphi, d-1}^j)$ is still a geodesic and we conclude that $J(\gamma_{\varphi, d-1}^j)=[\varphi (\gamma_{d-1}^j)]=\gamma_{\varphi, d-1}^j$.
Namely, a simple closed geodesic $\gamma_{\varphi, d-1}^j$ is symmetric with respect to the real line $\mathbb R$.
Therefore, we have $\gamma_{\varphi, d-1}^j=(\gamma_{\varphi, d-1}^j\cap\overline {\mathbb H})\cup J(\gamma_{\varphi, d-1}^j\cap\mathbb H)$ and ${}^{\#}(\mathbb R\cap \gamma_{\varphi, d-1}^j)=2$ as it is a simple closed curve.
Since the simple closed geodesic surrounds $\gamma_d^{2j-1}$ and $\gamma_{d}^{2j}$, it has to be homotopic to $\gamma_{d-1}^j$ and we have $\gamma_{\varphi, d-1  }^j=\gamma_{d-1}^j$.
This implies that $\varphi (\gamma_{d-1}^j)$ is homotopic to $\gamma_{d-1}^j$ and
the mapping $\varphi$ keeps the homotopy class of a pants curve of depth $d-1$ fixed.

Continuing the above argument, we conclude that $\varphi$ keeps the homotopy class of every pants curve fixed.
Therefore, if $\varphi$ is not homotopic to the identity on $X(\omega)$, the mapping class $[\varphi]$ has to contain a Dehn twist with respect to some pants curve.

Suppose that $\varphi$ contains a Dehn twist $\tau$ with respect to a pant geodesic $\gamma$.
For simplicity, we assume the curve $\gamma$ is the unit circle $\{|z|=1\}$. We take an annulus $A_R =\{0<R^{-1}<|z|<R\}$  for $R$ $(>1)$ which is sufficiently close to $1$. Then,  the Dehn twist $\tau$ could be given by
\begin{equation*}
	\tau : z\mapsto 
	\begin{cases}
		\exp \left \{ \log |z|+ \frac{2\pi i}{2\log R}(\log |z|+\log R)\right \}, &z\in A_R \\
		z, & z\not\in A_R.
	\end{cases}
\end{equation*}
Then, for $z\in A_R$,
\begin{equation*}
	\tau\circ J(z)=\exp \left \{ \log |z|+ \frac{2\pi i}{2\log R}(\log |z|+\log R)\right \}
\end{equation*}
and
\begin{equation*}
	J\circ \tau (z)= \exp \left \{ \log |z|- \frac{2\pi i}{2\log R}(\log |z|+\log R)\right \}.
\end{equation*}
Therefore, we conclude that $[\varphi\circ J]\not =[J\circ\varphi]$ and $\varphi$ does not contain Dehn twists if $[\varphi\circ J]=[J\circ \varphi]$.
So, we see that $[\varphi]\in \textrm{Mod}(X(\omega))_F\cap\textrm{Ker}(\Theta)$ is the identity and $\textrm{Mod}(X(\omega))_F$ is isomorphic to $F$ via $\Theta$.

\medskip
\noindent
{\bf $\textrm{Mod}(X(\omega))_T$. }
We define the subgroup $\textrm{Mod}(X(\omega))_T$ by
\begin{equation*}
	\textrm{Mod}(X(\omega))_T=<\textrm{Mod}(X(\omega)))_F, [\Phi_2]>=<[\Phi_0], [\Phi_1], [\Phi_2]>,
\end{equation*}
where $\Phi_2$ is the map given in \S4.
Since $\Theta ([\Phi_i])=f_1$ $(i=0, 1, 2)$, we have $\Theta (\textrm{Mod}(X(\omega))_T)=T$.

To show that $\textrm{Mod}(X(\omega))_T$ is isomorphic to $T$, we use the presentations in Proposition \ref{Pro:Presentaion}.
The homomorphism $\Theta$ sends $[\Phi_i]$ to $f_i$ for $i=0, 1, 2$.
Hence, if the generators $[\Phi_0], [\Phi_1]$ and $[\Phi_2]$ satisfy the 6 relations in Propostion \ref{Pro:Presentaion} as $T$, then we conclude that $\textrm{Mod}(X(\omega))_T$ is isomorphic to $T$ via $\Theta$.

The first two relations in Proposition \ref{Pro:Presentaion} come from the relations of $F$.
As we have seen that $\textrm{Mod}(X(\omega))_F$ is isomorphic to $F$ via $\Theta$, we verify that $[\Phi_0]$ and $[\Phi_1]$ satisfy the two relations.

We show that $[\Phi_2]^3=[id.]$ corresponding to (6) in Proposition \ref{Pro:Presentaion}.
The following Figure \ref{(6)} shows it.
\begin{figure}[htbp]
	\center
	\includegraphics[width=12cm]{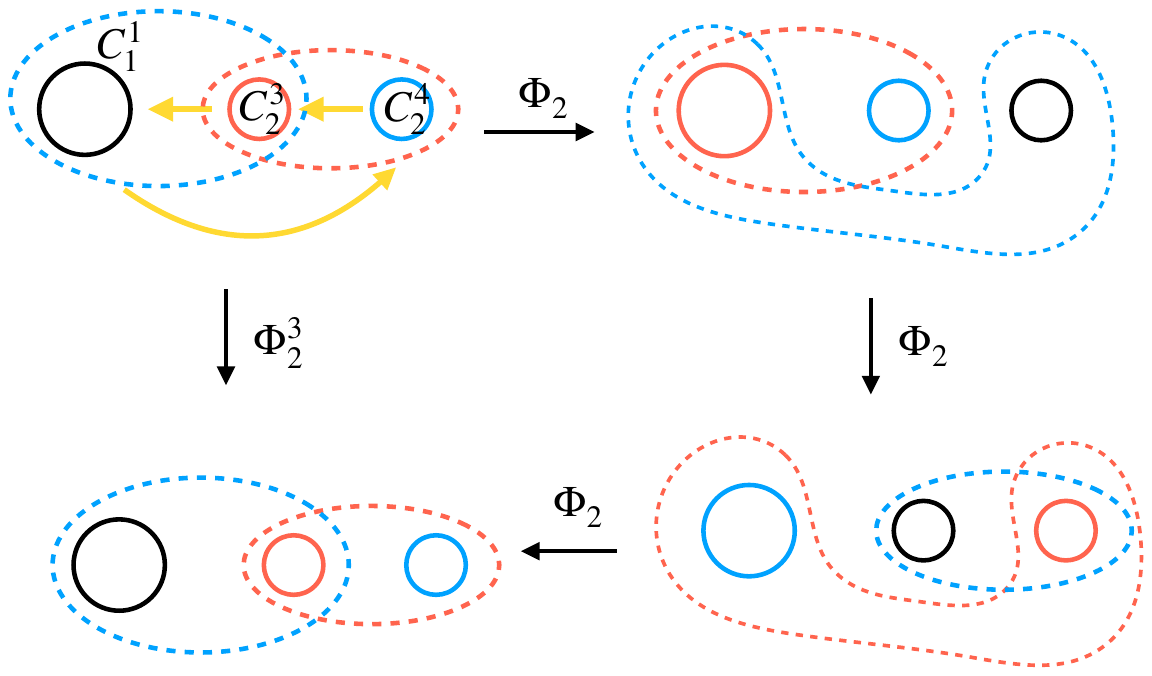}
		\caption{The relation (6) }
	\label{(6)}
\end{figure}
In the figure, yellow arrows indicate the directions of moves of pants curves $C_1^1, C_2^3, C_2^4$ by $\Phi_2$, and we give the same color to circles if they are corresponding to each other by $\Phi_2$.
Every pants curve inside $C_1^1, C_2^3, C_2^4$ is mapped to a pants curve while keeping the order by $\Phi_2$.
As we have noted in Remark \ref{Rem:NoTwists2}, the mapping $\Phi_2$ does not have Dehn twists with respect to $C_1^1, C_2^3, C_2^4$, and with respect to pants curves deeper than them.
Hence, $\Phi_2^3$ also does not have them.
Moreover, as Figure \ref{(6)} shows, $\Phi_2^3$ does not have Dehn twist with respect to the blue dotted curve and the red dotted curve.
Thus, we conclude that $\Phi_2^3$ is homotopic to the identity on $X(\omega)$.

Next, we see that $[\Phi_2^{-1}\circ\Phi_1\circ (\Phi_0^{-1}\circ\Phi_2\circ\Phi_1)]=[id.]$ corresponding to the relation (3) in Proposition \ref{Pro:Presentaion}.
We consider maps $\Phi_0, \Phi_1 : X(\omega)\to X(\omega)$ as in Figure \ref{Phi_0}.
\begin{figure}[htbp]
	\center
	\includegraphics[width=12cm]{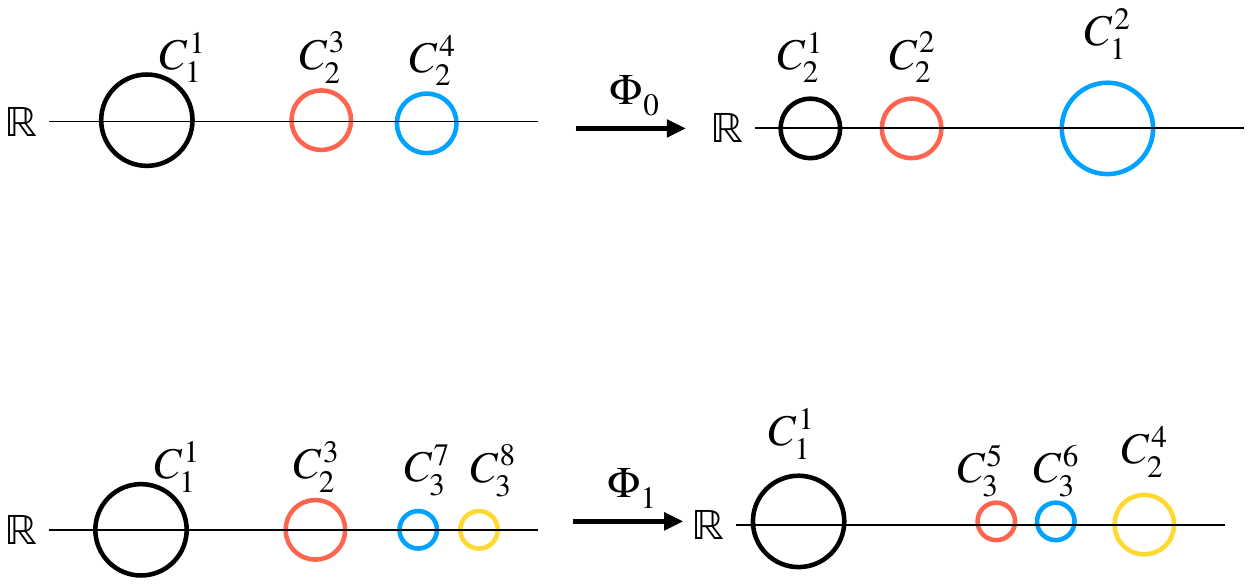}
		\caption{}
	\label{Phi_0}
\end{figure}
We may assume that $\Phi_0\circ J=J\circ \Phi_0$ and $\Phi_1\circ J=J\circ \Phi_1$, namely, they are symmetric.

We see that $\Phi_0^{\pm 1}$ and $\Phi_1^{\pm 1}$ change the depths of pants curves but they keep the order of the pants curves from the left to the right.
 On the other hand, the map $\Phi_2$ deforms a simple closed curve surrounding $C_1^1$ and $C_2^3$ as in Figure \ref{(6)} (see the dotted blue curves).
From (3) in Proposition \ref{Pro:Presentaion} and from the construction of $\Theta$, the mapping $\Phi_2^{-1}\circ\Phi_1\circ (\Phi_0^{-1}\circ\Phi_2\circ\Phi_1)$ fix every pants curve of sufficiently large depth.
Hence, we see that the mapping $\Phi_2^{-1}\circ\Phi_1\circ (\Phi_0^{-1}\circ\Phi_2\circ\Phi_1)$ finally keeps  \emph{all} pants curves fixed as Figure 
\ref{Fig(3)} shows.
\begin{figure}[htbp]
	\center
	\includegraphics[width=12cm]{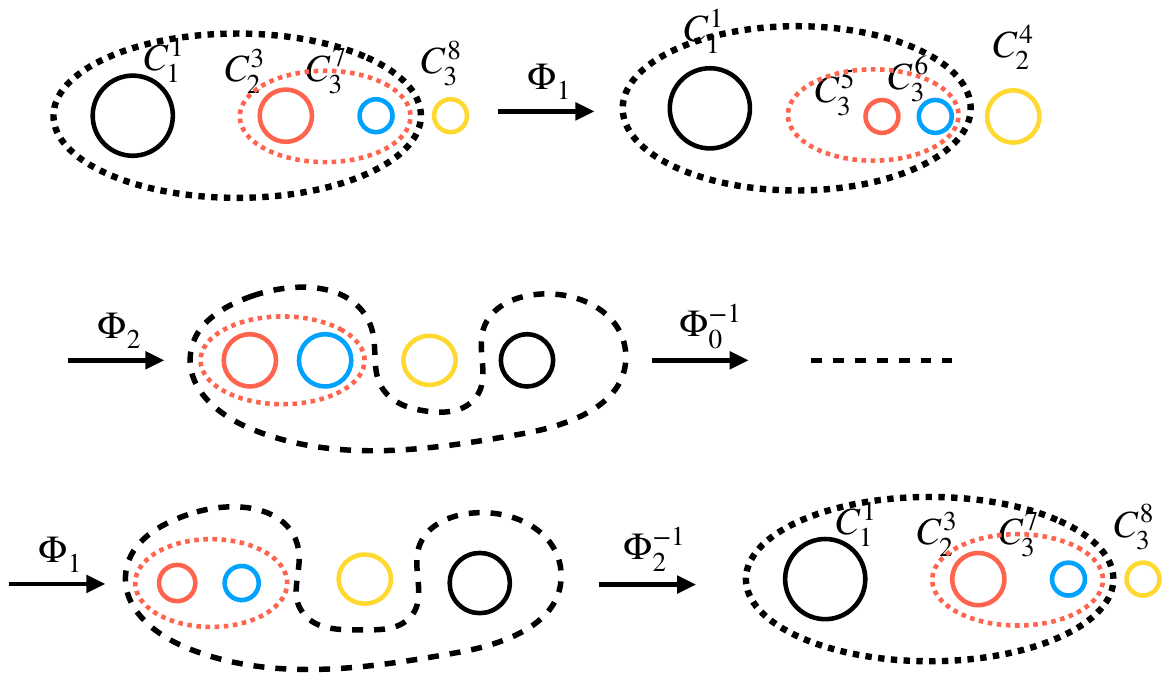}
		\caption{}
	\label{Fig(3)}
\end{figure}
Furthermore, by the same reason as $\Phi_2^3$, we see that the mapping does not have Dehn twists with respect to any pants curves.
Hence, the mapping is homotopic to the identity.

The above argument also works to show that the relation (4) is valid in $\textrm{Mod}(X(\omega))_T$.
Indeed, the relation (4) contains $\Phi_2$ and $\Phi_2^{-1}$ exactly once.
So, we are convinced that the relation (4) holds for $\textrm{Mod}(X(\omega))_T$ from the above argument.

As for the relation (5), we see that it is equivalent to $[\Phi_0\circ \Phi_2]^2=[id.]$ (cf. \cite{CannonFloydParry} p. 236).
This is shown directly by Figure \ref{(5)}.
Here, we also see that $(\Phi_0\circ \Phi_2)^2$ does not have Dehn twists with respect to any pants curve by using the same reason as above.
\begin{figure}[htbp]
	\center
	\includegraphics[width=12cm]{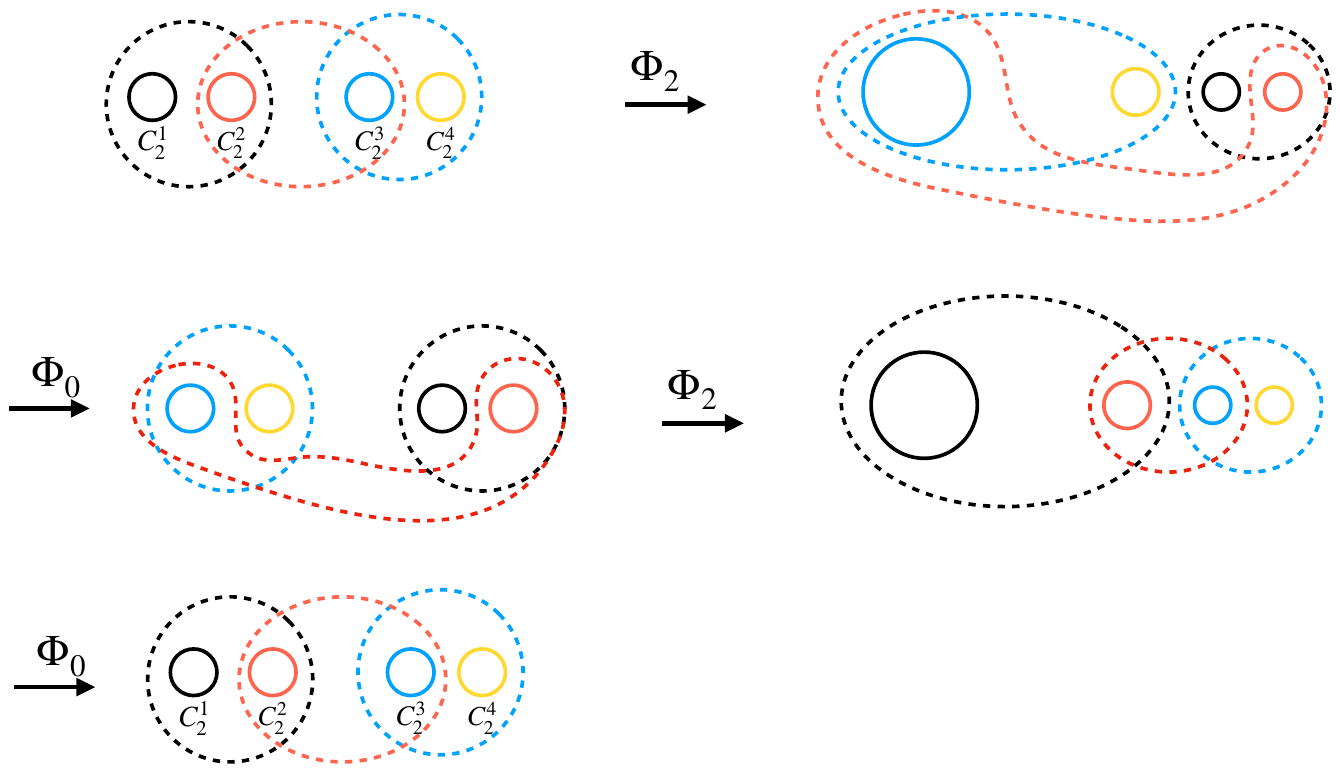}
		\caption{}
	\label{(5)}
\end{figure}

After all, we are convinced that all relations in Proposition \ref{Pro:Presentaion} hold for $[\Phi_0], [\Phi_1]$ and $[\Phi_2]$, the generators of $\textrm{Mod}(X(\omega))_T$ corresponding to $f_0, f_1$ and $f_2$ via $\Theta$.
Thus, we conclude that $\textrm{Mod}(X(\omega))_T$ is isomorphic to $T$ via $\Theta$.

\section{Proof of Theorem \ref{Thm:Discrete Action}}
We show that $\textrm{Mod}(X(\omega))_F$ and $\textrm{Mod}(X(\omega))_T$ act properly discontinuously on $
(\textrm{Teich}(X(\omega)), d_T)$.

Since $\textrm{Mod}(X(\omega))_F\subset \textrm{Mod}(X(\omega))_T$, it suffices to show that $\textrm{Mod}(X(\omega))_T$ acts properly discontinuously.
The following proposition essentially shows the discreteness of the action of the group on the Teichm\"uller space.

\begin{Pro}
\label{Pro:Key}
Suppose that $\omega$ satisfies the BRD condition.
	For $K\geq1$, we define a subset $\Delta (K)$ of $\textrm{Mod}(X(\omega))_T$ by
	\begin{equation*}
		\Delta (K)=\{[\varphi]\in \textrm{Mod}(X(\omega))_T \mid K([\varphi])<K\},
	\end{equation*}
	where 
	\begin{equation*}
		K([\varphi])=\inf \{K(\psi) \mid \psi\in [\varphi]\}.
	\end{equation*}
	Then, ${}^{\#}\Delta (K)<\infty$.
\end{Pro}
\begin{proof}
	Let $[\varphi]\in \textrm{Mod}(X(\omega))_T$ be in $\Delta (K)$.
Then, there exists a quasiconformal mapping $\phi : X(\omega)\to X(\omega)$ in $[\varphi]$ of positive order such that $K(\phi)<K$.
Let $d(K)\in \mathbb N$ be the quantity given in Lemma \ref{lemma: Key}.
Then, the map $\phi$ sends any pants curve of depth $d\geq d(K)$ to a pants curve, and the positively ordered tree diagram $(D_{[\varphi]}^{d(K)}, R_{[\varphi]}^{d(K)})$ for $[\varphi]$ obtained in \S 4 determines $\Theta ([\varphi])$.
Note that $D_{[\varphi]}^{d(K)}$ is determined by $\partial W_{d(K)}$, and $R_{[\varphi]}^{d(K)}$ is determined by $[\phi (\gamma_{d(K)}^1)], [\phi (\gamma_{d(K)}^2)], \dots , [\phi (\gamma_{d(K)}^{2^{d(K)}})]$.

For each $d\in \mathbb N$, we set
\begin{equation*}
	M_{d}=\max_{j=1, 2, \dots , 2^{d}}\ell_{X(\omega)}(\gamma_{d}^j),
\end{equation*}
and
\begin{equation*}
	m_{d}=\min_{j=1, 2, \dots , 2^{d}}\ell_{X(\omega)}(\gamma_{d}^j).
\end{equation*}
We define a quantity $N(K)$ by
 \begin{equation*}
	N(K):={}^\#\left \{(d, j)\in \mathbb N\times\mathbb N \mid \frac{1}{K}m_{d(K)}\leq \ell_{X(\omega)}(\gamma_d^j)\leq KM_{d(K)}\right \}.
\end{equation*}
Since $\lim_{d\to\infty}\ell_{X(\omega)}(\gamma_d^j)=0$, we see that $N(K)$ is finite for each $K (\geq 1)$.

On the other hand, it follows from Wolpert's lemma that
\begin{eqnarray*}
	\frac{1}{K}m_{d(K)}&\leq & \frac{1}{K}\ell_{X(\omega)}(\gamma_{d(K)}^j) \leq \ell_{X(\omega)}([\phi (\gamma_{d(K)}^j)]) \\
	&\leq &K\ell_{X(\omega)}(\gamma_{d(K)}^j)\leq KM_{d(K)}.
\end{eqnarray*}

Therefore, the number of possibilities of $[\phi (\gamma_{d(K)}^j)]$ is at most $N(K)$ $(j\in \{1, 2, \dots , 2^{d(K)}\})$, and the number of possibilities of $R_{[\varphi]}^{d(K)}$ is at most $N(K)^{2^{d(K)}}$.
This means that the number of possibilities of $\Theta ([\varphi])$ for $[\varphi]\in \Delta (K)$ is at most $N(K)^{2^{d(K)}}$.
Since $\Theta : \textrm{Mod}(X(\omega))_T\to T$ is injective, we have ${}^{\#}\Delta (K)\leq N(K)^{2^{d(K)}}<\infty$.
\end{proof}

Now, we show that the action of $\textrm{Mod}(X(\omega))_T$ on $(\textrm{Teich}(X(\omega)), d_T)$ is properly discontinuous.

For $\varepsilon >0$ and $P=[X, f]\in \textrm{Teich}(X(\omega))$, we put
\begin{equation*}
	B(P, \varepsilon)=\{Q\in \textrm{Teich}(X(\omega)) \mid d_T(P, Q)<\varepsilon\}.
\end{equation*} 
Suppose that $\chi_{[\varphi]}(P)\in B(P, \varepsilon)$ for $[\varphi]\in \textrm{Mod}(X(\omega))_T$.
Then, we see that the inequality 
\begin{equation}
\label{eqn:TriTeich}
	d_T(P_0, \chi_{[\varphi]}(P_0))\leq d_T(P_0, P)+d_T(P, \chi_{[\varphi]}(P))+d_T(\chi_{[\varphi]}(P), \chi_{[\varphi]}(P_0))
\end{equation}
holds for $P_0=[X(\omega), id.]$.
Since $\chi_{[\varphi]}$ is an isometry with respect to the Teichm\"uller distance, we have
\begin{equation}
\label{eqn:IsomTeich}
	d_T(\chi_{[\varphi]}(P), \chi_{[\varphi]}(P_0))=d_T(P, P_0).
\end{equation}
It follows from (\ref{eqn:TriTeich}) and (\ref{eqn:IsomTeich}) that
\begin{equation*}
	d_T (P_0, \chi_{[\varphi]}(P_0))\leq \varepsilon+2d_T(P_0, P).
\end{equation*}
By the definition of the Teichm\"uller distance, $d_T(P_0, \chi_{[\varphi]}(P_0))=\log K([\varphi])$.
Thus, we see that $[\varphi]\in \Delta (\exp \{\varepsilon+2d_T(P_0, P)\})$.
Therefore, it follows from Proposition \ref{Pro:Key} that the set
\begin{equation*}
	\{[\varphi]\in \textrm{Mod}(X(\omega))_T \mid \chi_{[\varphi]}(P)\in B(P, \varepsilon)\}
\end{equation*}
is finite.
We conclude that the action of $\textrm{Mod}(X(\omega))_T$ is properly discontinuous on $(\textrm{Teich}(X(\omega)), d_T)$.

\section{Proof of Theorem \ref{Thm:indiscrete} }
From the definition of $\Pi : \textrm{Mod}(X(\omega))\to \textrm{pre-Mod}(X(\omega))$, we see that $\Pi ([\varphi])=\Pi ([\psi])$ if and only if $\varphi|_{E(\omega)}=\psi|_{E(\omega)}$.
It follows from Theorem \ref{Thm:Exact} that
\begin{eqnarray*}
	\textrm{pre-Mod}(X(\omega))_F=\Pi (\textrm{Mod}(X(\omega))^{OP}) \\
	 \simeq \textrm{Mod}(X(\omega))^{OP}/\textrm{Mod}(X(\omega))_0 \simeq F, \\
	\textrm{pre-Mod}(X(\omega))_T=\Pi (\textrm{Mod}(X(\omega))^{PO}) \\
	 \simeq \textrm{Mod}(X(\omega))^{PO}/\textrm{Mod}(X(\omega))_0 \simeq T
\end{eqnarray*}
and
\begin{eqnarray*}
	\textrm{pre-Mod}(X(\omega))_V=\Pi (\textrm{Mod}(X(\omega))^{POP}) \\
	 \simeq \textrm{Mod}(X(\omega))^{POP}/\textrm{Mod}(X(\omega))_0 \simeq V.
\end{eqnarray*}
Indeed, the isomorphisms $\theta_G$ from $\textrm{pre-Mod}(X(\omega))_G$ to $G$ $(G=F, T, V)$ are given by
\begin{equation}
\label{Def:pre-iso}
	\theta_G (\langle \varphi \rangle ):=\Theta ([\phi])
\end{equation}
for $\phi\in \Pi^{-1}(\langle \varphi\rangle )$.
The map $\theta_G$ is well-defined.

Now, we consider the discreteness of the action on $(\textrm{pre-Teich}(X(\omega)), d_{pre-T})$.
At first, we show that $\textrm{pre-Mod}(X(\omega))_T$ acts properly discontinuously on the space.

For $\varepsilon >0$ and $P=\langle X, f\rangle\in \textrm{pre-Teich}(X(\omega))$, we put
\begin{equation*}
	B_{pre} (P, \varepsilon)=\{Q\in \textrm{pre-Teich}(X(\omega))\mid d_{pre-T}(P, Q)<\varepsilon\}.
\end{equation*}
By using the same argument as in \S6, we see that if $\tilde\chi_{\langle\varphi\rangle}(P)\in B_{pre}(P, \varepsilon)$, then
\begin{equation}
\label{eqn:pre-bound}
	d_{pre-T}(P_0, \tilde\chi_{\langle\varphi\rangle}(P_0))\leq \varepsilon+2d_{pre-T}(P_0, P),
\end{equation}
for $P_0=\langle X(\omega), id.\rangle$.

\begin{Pro}
\label{Pro:pre-key}
	Suppose that $\omega$ satisfies the BRD condition.
	For $K\geq1$, we define a subset $\Delta (K)_{pre}$ of $\textrm{pre-Mod}(X(\omega))_T$ by
	\begin{equation*}
		\Delta (K)_{pre}:=\{\langle\varphi\rangle\in \textrm{pre-Mod}(X(\omega))_T \mid K(\langle\varphi\rangle)<K\},
	\end{equation*}
	where 
	\begin{equation*}
		K(\langle\varphi\rangle)=\inf \{K(\psi) \mid \Pi (\psi)= \Pi(\varphi)\}.
	\end{equation*}
	Then, ${}^{\#}\Delta (K)_{pre}<\infty$.
\end{Pro}
\begin{proof}
	The proof is similar to that of Proposition \ref{Pro:Key}.
	So, we just give a brief sketch of the proof.
	
	Let $\langle\varphi\rangle$ be in $\Delta_{pre}(K)$.
	Then, there exists a quasiconformal mapping $\phi : X(\omega)\to X(\omega)$  such that $[\phi]\in\Pi^{-1} (\langle\varphi\rangle)$ and $K(\phi)<K$.
	Then, by using the same argument in the proof of Proposition \ref{Pro:Key}, we see that the number of possibilities of $\Theta ([\phi])$ is less than a number depending only on $K$.
	From (\ref{Def:pre-iso}), it seen that $\theta_T(\Delta_{pre}(K))\subset T$ is a finite subset.
	Since $\theta_T$ is an injection, we conclude that $\Delta_{pre}(K)$ is also a finite set. 
\end{proof}
Thus, Proposition \ref{Pro:pre-key} together with (\ref{eqn:pre-bound}) shows the properly discontinuousness of the action of $\textrm{pre-Mod}(X(\omega))_T$ on $\textrm{pre-Teich}(X(\omega)), d_{pre-T})$.

\medskip

\noindent
{\bf The action of $\textrm{pre-Mod}(X(\omega))_V$.}
First of all, we shall explain the idea of the proof.
We have already known that the hyperbolic lengths of pants geodesics with large depth are short (Proposition \ref{Pro:Length to 0}). 
Therefore, such geodesics have collars with large width, and the \lq \lq half twist\rq\rq around the geodesic is achieved by a quasiconformal mapping with small dilatation.
So, if we may find a sequence of such half twists coming from mappings of $V$, then we get the result.
That is the idea of the proof.

We start with looking for mappings in $V$ which give a half twist around a deep pants geodesic. 
Since $\lim_{n\to \infty}q_n=1$, there exists $N\in \mathbb N$ such that $q_n>1/2$ if $n>N$.
For $n >N$, we consider $g_n:=f_0^{n-1}\circ g\circ f_0$ and $h_n:=f_0^n$, where $f_0\in F$ is a mapping given in Example \ref{Ex.F} and $g=\Theta ([\phi_3])\in V$ is given in \S6.
Then, tree diagrams for $g_n$ and $h_n$ are given by Figure \ref{Fig.TreeDiagg_nh_n}.
\begin{figure}[htbp]
	\center
	\includegraphics[width=12cm]{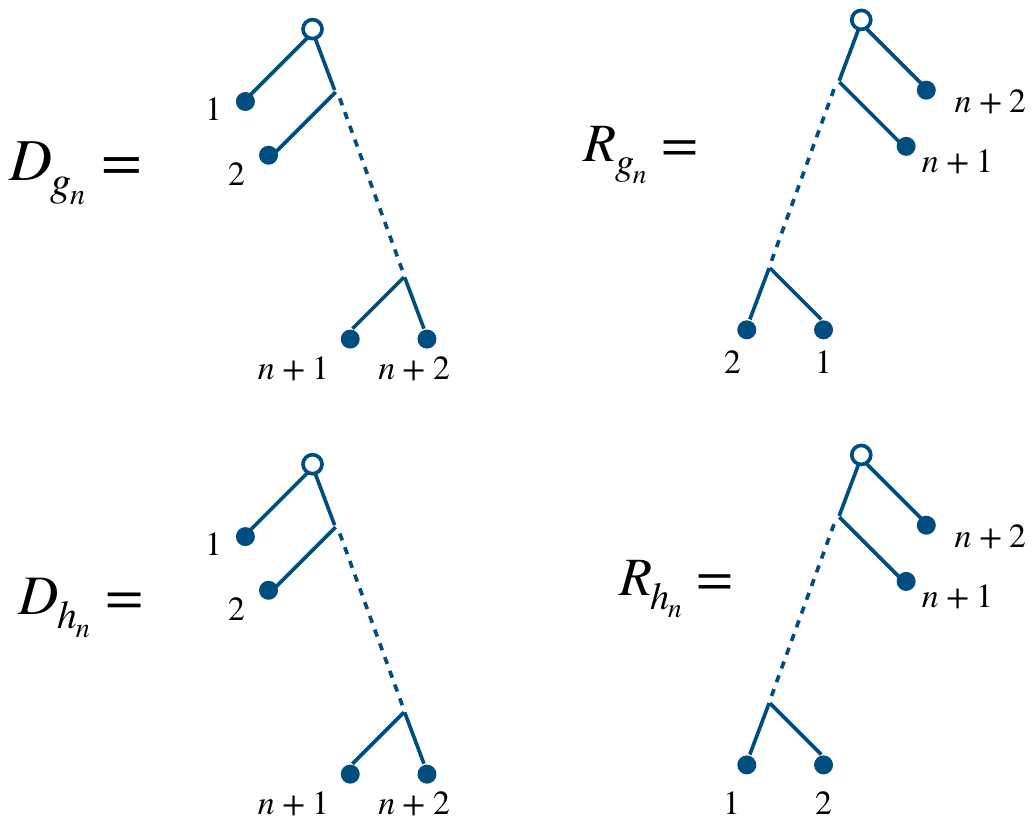}
		\caption{ }
	\label{Fig.TreeDiagg_nh_n}
\end{figure}
Therefore, a tree diagram for $g_n\circ h_n^{-1}=f_0^{n-1}\circ g\circ f_0^{-n+1}$ is given by Figure \ref{Fig.TreeDiagPhi}.
\begin{figure}[htbp]
	\center
	\includegraphics[width=12cm]{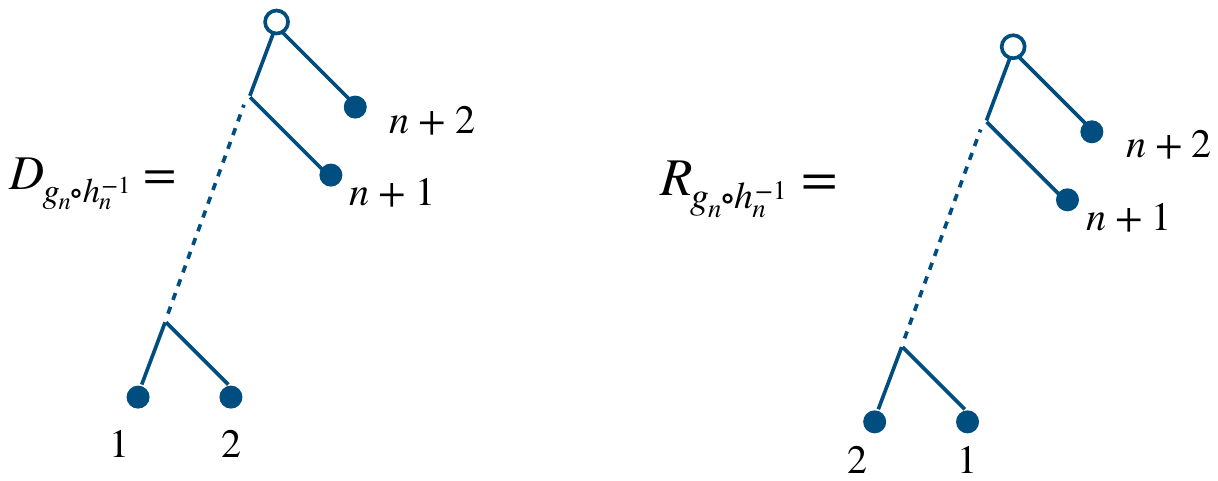}
		\caption{ }
	\label{Fig.TreeDiagPhi}
\end{figure}

Now, we construct a quasiconformal mapping $\varphi_n : X(\omega)\to X(\omega)$ such that $\langle\varphi_n\rangle\in \textrm{pre-Mod}(X(\omega))_V$ and $\theta _V(\langle\varphi_n\rangle)=g_n\circ h_n^{-1}$.
We also show that the mapping $\varphi_n$ represents a half twist around a deep pants geodesic and the dilatation is small.

 First, we look at the closed intervals $I_{n+1}^1, I_{n+1}^2$ which appear in the $(n+1)$-th step of the construction of $E(\omega)$.
 Actually, $I_{n+1}^1=[0, 2^{-1}(1-q_{n+1})|I_n^1|]$ and $I_{n+1}^2=[2^{-1}(1+q_{n+1})|I_n^1|, |I_n^1|]$.
 We consider two disks $\Delta_0^n$ and $\Delta_1^n$. The disk $\Delta_0^n$ is  bounded by a circle  centered at $a_{0n}$, the midpoint of $I_{n}^1$, with radius $\rho_{0n}:=|I_{n}^1|/2$, and the disk $\Delta_1^n$ is bounded by a circle centered at $a_{0n}$ with radius $\rho_{1n}:=\rho_{0n}+2^{-1}q_{n}|I_{n-1}^1|$. 
 Since $|I_n^1|=2^{-1}(1-q_n)|I_{n-1}^1|$ and $\lim_{n\to \infty}q_n=1$, we have
 \begin{equation*}
 	\frac{\rho_{0n}}{\rho_{1n}}=\frac{1-q_n}{(1-q_n)+2q_n}\to 0,
 \end{equation*}
 as $n\to +\infty$.
 
 We construct a quasiconformal mapping $\psi_{0n} : X(\omega)\to X(\omega)$ which exchanges $I_{n+1}^1$ and $I_{n+1}^2$ by rotation on $\Delta_0^n$ and fixes any point in $\mathbb C\setminus \Delta_1^n$.
 In fact, the mapping $\psi_{0n}$ is defined by
 \begin{equation*}
 	\psi_{0n} (z)=\begin{cases}
 		z, & z\in \mathbb C\setminus \Delta_1^n, \\
 		a_{0n} + (z-a_{0n})\exp \left\{\pi i \frac{\rho_{1n}-|z-a_{0n}|}{\rho_{1n}-\rho_{0n}}\right\}, & z\in \Delta_1^n\setminus \Delta_0^n, \\
 		-z+2a_{0n}, & z\in \Delta_0^n.
 	\end{cases}
 \end{equation*}
 Since $\rho_{0n}/\rho_{1n}\to 0$ $(n\to +\infty)$, we see that $K(\psi_{0n})\to 1$ $(n\to +\infty)$.
 
The mapping $\psi_{0n}$ exchanges $I_{n+1}^1$ and $I_{n+1}^2$ but it is not order preserving.
We make another mapping $\psi_{1n}$ so that $\psi_{1n}\circ \psi_{0n}$ is order preserving.
 
  Let $\Delta_{jk}^n$ $(j, k=1, 2)$ be four disks centered at $b_{kn}$, the midpoint of $I_{n+1}^k$, with radius $\delta_{jn}:=|I_{n+1}^1|/2+ q_{n+1}(j-1)|I_n^1|/3$.
  So, $U_k^n:=\Delta_{2k}^n\setminus\overline{\Delta_{1k}^n}$ $(k=1, 2)$ are annuli surrounding $I_{n+1}^k$.
  
  The length of the open interval between $I_{n+1}^1$ and $I_{n+1}^2$ is $q_{n+1}|I_n^1|$, and the  length of the open interval between $I_{n}^1$ and $I_{n}^2$ is $q_n|I_{n-1}^1|$.
  Since $q_n, q_{n+1}>1/2$, we have
    \begin{eqnarray*}
  	q_n|I_{n-1}^1|-q_{n+1}|I_n^1|&=&q_n|I_{n-1}^1|-\frac{1-q_n}{2}q_{n+1}|I_{n-1}^1| \\
  	&=&(2-(1-q_n)(2+q_{n+1}))\frac{|I_{n-1}^1|}{2} 
  	> 0.
  \end{eqnarray*}
 Hence, $E(\omega)\cap (U_1^n\cup U_2^n)=\emptyset$.
 We also see that $\delta_{1n}/\delta_{2n}\to 0$
 
 Then, we define the mapping $\psi_{1n}$ by
 \begin{equation*}
 	\psi_{1n} (z)=\begin{cases}
 		z, &z\in \mathbb C\setminus (\Delta_{21}\cup\Delta_{22}), \\
 		b_{kn}+(z-b_{kn})\exp \left\{\pi i\frac{\delta_{22}-|z-b_{kn}|}{\delta_{2n}-\delta_{1n}}\right \}, &z\in U_k^n\quad  (k=1, 2), \\
 		-z +2b_{kn}, &z\in\Delta_{1k}\quad  (k=1, 2). \\ 
 	\end{cases}
 \end{equation*}
 Since $\delta_{1n}/\delta_{2n}\to 0$, we see that $K(\psi_{1n})\to 1$ as $n\to +\infty$.
Therefore, for $\varphi_n :=\psi_{1n}\circ \psi_{0n}$, we have $K(\varphi_n)\to 1$ because $K(\varphi_n)\leq K(\psi_{0n})K(\psi_{1n})$.
Hence, $d_{T}(P_0, \chi_{[\varphi_n]}(P_0))\to 0$ as $n\to +\infty$.

From the construction, the mapping $\varphi_n$ represents a mapping in $V$ corresponding to the tree diagram $(D_{g_n\circ h_n^{-1}}, R_{g_n\circ h_n^{-1}})$.
Hence, we conclude that $\langle\varphi_n\rangle\in \textrm{pre-Mod}(X(\omega))_V$, and we complete the proof of the theorem.
 
\section{Examples}

We have shown that $\textrm{Mod}(X(\omega))_T$ acts properly discontinuously on the Teichm\"uller space $\textrm{Teich}(X(\omega))$ if $\omega\in \Omega$ satisfies the BRD condition.
Since $\textrm{Mod}(X(\omega))_T$ is isomorphic to Thompson's group $T$ via $\Theta$, we may say that the action of Thompson's group $T$ on $\textrm{Teich}(X(\omega))$ is properly discontinuous.
In this section, we show that there are \emph{infinitely many} distinct Teichm\"uller spaces of generalized Cantor sets where the action of Thompson's group $T$ is properly discontinuous.

From a result of our previous paper \cite{ShigaStructure}, we obtain the following.
\begin{Pro}
\label{Pro:Distinct}
	Suppose that $\omega=(q_n)_{n=1}^{\infty}$, $\omega'=(q_n')_{n=1}^{\infty}\in \Omega$ are increasing, that is, $q_{n}\leq q_{n+1}, q_{n}'\leq q_{n+1}'$ $(n\in \mathbb N)$.
	If $\lim_{n\to \infty}q_n=\lim_{n\to \infty}q_n'=1$ and
	\begin{equation*}
		\sup_{n\in \mathbb N}\left |  \log \frac{\log (1-q_n)}{\log (1-q_n')}\right |=\infty,
	\end{equation*}
	then $X(\omega)$ and $X(\omega')$ are not quasiconformally equivalent to each other.
	In particular, $\textrm{Teich}(X(\omega))\not=\textrm{Teich}(X(\omega'))$.
\end{Pro}

We define functions $\log^{(1)} x$ on $\mathbb R_{>0}$ by $\log ^{(1)}x=\log x$.
Functions $\log ^{(k)}x$ $(k\in \mathbb N)$ are inductively defined by
\begin{equation*}
	\log^{(k+1)}x=\begin{cases}
		1, &0<x\le e_{k} \\
		\log \log^{(k)}x, &e_k<x
	\end{cases}
\end{equation*}
where $e_{k}$ is the positive number with $\log ^{(k)}e_k=e$.

We take $\omega_k=(q_n^{(k)})_{n=1}^{\infty}\in \Omega$ for each $k\in \mathbb N$ as
\begin{equation*}
	q_n^{(k)}=1-\frac{1}{2\log^{(k)}n}, n\in \mathbb N.
\end{equation*}
Then, $\omega_k$ are increasing $(k=1, 2, \dots)$ and if $k'>k$, then we have
\begin{equation*}
	\lim_{n\to \infty} \frac{\log (1-q_n^{(k')})}{\log (1-q_n^{(k)})}=\lim_{n\to \infty}\frac{\log ^{(k'+1)}n}{\log^{(k+1)}n}=\lim_{N\to\infty}\frac{\log^{(k'-k)}N}{N}=0.
\end{equation*}
It follows from Proposition \ref{Pro:Distinct} that $\textrm{Teich}(X(\omega_k))\not=\textrm{Teich}(X(\omega_{k'}))$ if $k\not=k'$.

On the other hand, it is not hard to see that each $\omega_k$ satisfies the BRD condition.
Thus, we verify that Thompson's group $T$ acts properly discontinuously on all $\textrm{Teich}
(X(\omega_k))$ $(k=1, 2, \dots )$.

\end{document}